\documentclass[12pt,a4paper,leqno]{article}

\usepackage[utf8]{inputenc}
\usepackage[T1]{fontenc}
\usepackage[english]{babel}
\usepackage{amsmath}
\usepackage{amsfonts}
\usepackage{amssymb}
\usepackage{multicol}
\usepackage{color}
\usepackage{soul}
\usepackage{dsfont}
\usepackage{pdfpages}
\usepackage{graphicx}
\usepackage{listings}
 \usepackage{hyperref}
\usepackage{fancyhdr}
\usepackage{algorithm}
\usepackage{comment} 
\usepackage{mathabx} 
\usepackage{algorithmic}

\newcommand{\R}{\mathbb{R}}

\newcommand{\lp}{\left(}
\newcommand{\rp}{\right)}
\newcommand{\N}{\mathbb{N}}

\newcommand{\Sc}{\mathcal S(\R)}
\newcommand{\Scd}{\mathcal S(\R^d)}
\newcommand{\I}{\leqslant}
\newcommand{\eps}{\varepsilon}
\newcommand{\s}{\geqslant}
\newcommand{\E}{\mathbb{E}}

\newcommand{\PR}{\mathbb{P}}
\newcommand{\tem}{\mathcal S'(\R)}
\newcommand{\temd}{\mathcal S'(\R^d)}
\newcommand{\m}{\mathcal}
\newcommand{\pam}{\textbf{PAM}}
\newcommand{\dd}{\, \mathrm{d}}

\numberwithin{equation}{section}
\usepackage{ccaption}
\usepackage{amsthm}
\newtheorem{theo}{Theorem}[section]
\newtheorem{prop}[theo]{Proposition}
\newtheorem{rem}[theo]{Remark}
\newtheorem{df}[theo]{Definition}
\newtheorem{lem}[theo]{Lemma}
\newtheorem{cor}[theo]{Corollary}

\newcommand{\scal}[2]{\left \langle #1 , #2 \right \rangle }

\newcommand{\limscd}[1]{\overset{\Scd}{\longrightarrow}}
\newcommand{\lims}[1]{\to}

\newcommand{\ssup}[1]{\underset{#1}{\sup}}
\newcommand{\leb}[2]{\text{Leb}_{#1}\lp #2 \rp}
\allowdisplaybreaks[1]
\begin{document}
%\maketitle
\begin{center}{\bf\Large L\'evy Processes and L\'evy White Noise
\vskip 12pt

 as Tempered Distributions
}
\vskip 16pt
{\bf Robert C.~Dalang}\footnote[1]{Institut de math\'ematiques,
Ecole Polytechnique F\'ed\'erale de Lausanne, Station 8,
CH-1015 Lausanne, Switzerland.
Emails: robert.dalang@epfl.ch, thomas.humeau@epfl.ch ~~ 

Partially supported by the Swiss National Foundation for Scientific Research.

\vskip 12pt


  {\em Abbreviated Title.} L\'evy Noise as a Tempered Distribution
}
and {\bf Thomas Humeau}$^1$
\vskip 16pt
Ecole Polytechnique F\'ed\'erale de Lausanne
\vskip 16pt

\end{center}

\begin{abstract}
 We identify a necessary and sufficient condition for a L\'evy white noise to be a tempered distribution. More precisely, we show that if the L\'evy measure associated with this noise has a positive absolute moment, then the L\'evy white noise almost surely takes values in the space of tempered distributions. If the L\'evy measure does not have a positive absolute moment of any order, then the event on which the L\'evy white noise is a tempered distribution has probability zero. 
 \end{abstract}
 
\section{Introduction}

%\indent\indent 
   It is well-known that Gaussian white noise in $\R^d$ is a generalized random field that can be viewed as a random element of the space $\temd$ of tempered (Schwartz) distributions \cite{gelfand,spdewalsh}. It is natural to ask whether the same is true of L\'evy white noise?

%\indent 
   This abstract mathematical question was posed to the authors by M.~Unser and J.~Fageot, who were interested in developing sparse statistical models for signal and image processing \cite{fageot}. For this, they considered generalized random fields with values in $\temd$. L\'evy white noises provide interesting examples of generalized random fields, and together with A.~Amini, they showed in \cite[Theorem 3]{fageot} that a sufficient condition for L\'evy white noise to take values in $\temd$ is that the associated L\'evy measure have a positive absolute moment. The main result of this paper is that this condition is, in fact, necessary and sufficient.

   The result of Unser and Fageot improves several other partial results that appear in the mathematical literature. In \cite{lee}, L\'evy white noise is studied as a natural generalization of Gaussian white noise, and the authors showed that this process takes values in $\temd$ if the associated L\'evy measure has a first absolute moment. In order to develop a white noise theory for L\'evy noise, Di Nunno et al.~\cite{dinunno} consider L\'evy white noise with a L\'evy measure that has a finite second moment. In \cite[Theorem 4.1]{shih}, Y.-L.~Lee and H.-H.~Shih give a necessary and sufficient condition for L\'evy white noise to take values in $\temd$; however, this condition involves checking the continuity of a functional and does not translate directly into a condition on the L\'evy measure. Finally, knowing that L\'evy white noise takes values in $\temd$ is useful in the study of stochastic partial differential equations driven by L\'evy noise, as in \cite{bernt}, which again considers the case where the L\'evy measure is square integrable.
	
	In order to address the question of Unser and Fageot, we first consider in Section \ref{levy_dim_one} the case of dimension $d=1$. In this case, L\'evy white noise can be viewed as the derivative of a L\'evy process, for which there is a large literature (see \cite{applebaum,cont,sato}, for instance). The question of whether or not sample paths of a L\'evy process belong to $\temd$ reduces essentially to whether or not this process is slowly growing (that is, has no more than polynomial growth: see Remark \ref{rkpolgrowth}). We make use of the L\'evy-It\^o decomposition $X_t=\gamma t +\sigma W_t + X^P_t+ X^M_t$ of a L\'evy process $(X_t)_{t\in \R_+}$, in which $(W_t)$ is a standard Brownian motion, $(X^P_t)$ is a compound Poisson process (term containing the large jumps of $X$), and $(X^M_t)$ is a square integrable pure-jump martingale (term containing the small jumps of $X$). Using the strong law of large numbers for L\'evy processes with a first moment, we show that the L\'evy process $(\gamma t +\sigma W_t +X^M_t)$ is always slowly growing, so the question reduces to the study of the process $(X^P_t)$. 
	
	A first result (Proposition \ref{growth}) is that a compound Poisson process can belong to $\temd$ if and only if it is slowly growing. The question now reduces to determining when a compound Poisson process is slowly growing, which is addressed in the literature (see \cite[Section 48]{sato}), but for which we give a direct answer using the law of large numbers of Kolmogorov, Marcinkiewicz and Zygmund \cite{kallenberg} (see Proposition \ref{longterm}).
	
	With the result for L\'evy processes in hand, we then easily deduce the corresponding result for L\'evy white noise (see Theorem \ref{tempered}). Since \cite{fageot} constructs L\'evy white noise as a measure on the cylinder $\sigma$-field of $\temd$ via the Bochner-Minlos theorem \cite{gelfand}, we relate our result to that of \cite{fageot} by taking care to show that L\'evy noise actually defines a random variable with values in $\temd$ equipped with its Borel $\sigma$-field (which is in fact equal to the cylinder $\sigma$-field, see the proof of Corollary \ref{measurable}).
	
	In Section \ref{levy_dim_d}, we turn to L\'evy random fields and L\'evy noise on $\R^d$, with $d \geq 1$. Again, in the case of a L\'evy random field, the L\'evy-It\^o decomposition applies (see \cite{adler,walshdalang}), and the three terms with moments greater than $1$ always have sample paths with values in $\temd$ (see Proposition \ref{martingaled}). For the term containing the large jumps, which is a compound Poisson sheet $X^P$,  we are hampered by the fact that even if there are multiparameter analogues of the law of large numbers of Kolmogorov, Marcinkiewicz and Zygmund (see \cite{klesov}), multiparameter random walks cannot be easily used to represent compound Poisson sheets. Therefore, we make use of our study in dimension $1$ by considering the L\'evy random field $X^P$ along a line parallel to a coordinate axis. This defines a (one-parameter) L\'evy process $L$. A key technical step is to identify (in Lemma \ref{sequence}) a sequence of test functions $(\varphi_n) \subset \Scd$ with polynomially growing norms such that $\langle X^P, \varphi_n\rangle$ give precisely the value of $L$ at the time of its $n$-th jump. This leads to the characterization of L\'evy white noises and random fields that take values in $\temd$ (see Theorem \ref{temperedd} and Corollary \ref{cor3.15}).

We now introduce the main notation that will be used throughout the paper. Let $d\in \N \setminus \{0\}$. For a multi-index $\alpha=(\alpha_1,...,\alpha_d)\in \N^d$, a smooth function $\varphi: \R^d\rightarrow \R$, and $t\in \R^d$, we write $t^\alpha=\prod_{i=1}^d t_i^{\alpha_i}$ and $\varphi^{(\alpha)}=\frac{\partial^{|\alpha|}\varphi}{\partial t_1^{\alpha_1}\cdot\cdot\cdot\partial t_d^{\alpha_d}}$, where $|\alpha|=\sum_{i=1}^d \alpha_i$. For $\beta, \gamma \in \N^d$, we write $\binom \beta \gamma=\frac{\beta !}{\gamma ! (\beta-\gamma)!}$, where $\beta !=\prod_{i=1}^d (\beta_i !)$. When $t\in \R^d$, we also write $|t|$ for the Euclidian norm and the meaning should be clear from the context. The Schwartz space is denoted $\Scd$ and is the space of all smooth functions $\varphi: \R^d \rightarrow \R$ such that, for all multi-indexes $\alpha$, $\beta \in \N^d$,  we have
 $$\ssup{t\in \R^d} \left | t^\alpha \varphi^{(\beta)}(t) \right | <+\infty \, .$$
 This space is equipped with the topology defined by the family of norms $\m N_p$, where, for all $p\in \N$ and $\varphi \in \Scd$,
 $$ \mathcal N_p(\varphi)= \underset{|\alpha |, |\beta | \I p}{\sum} \, \ssup{t\in \R^d} \left | t^\alpha \varphi^{(\beta)}(t) \right | \, .$$
 A basis of neighborhoods of the origin for this topology is given by the family 
\begin{equation}\label{neighborhood}
  \lp \left \{ \varphi \in \Scd :  \mathcal N_p(\varphi) < \varepsilon \right \} \rp_{p\in \N, \, \varepsilon >0} \, ,
\end{equation}
since such a basis is usually given by finite intersections of sets of this form, and for all $p\in \N, \ \varphi \in \Scd$, $\m N_p(\varphi) \I \m N_{p+1}(\varphi)$.  A sequence $(\varphi_n)_n$ converges to zero in $\Scd$ if for all $p\in \N, \ \mathcal N_p(\varphi_n)  \to 0$ as $n\to +\infty$. The space of tempered distributions is denoted $\mathcal S'(\R^d)$ and is the space of all continuous linear functionals on $\Scd$. Equivalently, $u\in \mathcal S'(\R^d)$ if and only if there is an integer $p\s 0$ and a constant $C$ such that for all $\varphi \in \Scd$, 
 $$ \left |\scal{u}{\varphi}\right | \I C \mathcal N_p(\varphi) \, .$$
% Indeed, if $u\in \temd$, by the definition of continuity and by the fact that \eqref{neighborhood} defines a basis of neighborhood of the origin, there exists $p\in \N$ and $\delta>0$, such that for any $\varphi \in \Scd$ satisfying $\m N_p(\varphi) < \delta$, Then  $|\scal u \varphi | \I 1$. By linearity we deduce that for all $\varphi \in \Scd, \ |\scal u \varphi | \I \frac 2 \delta \m N_p(\varphi)$.
\begin{rem}\label{rkpolgrowth}
 We say that a function $f:\R^d \to \R$ is slowly growing if $\sup_{t\in \R^d}|f(t)|(1+|t|)^{-\alpha}<\infty$ for some $\alpha\s 0$. In this case, $f$ defines a tempered distribution by the formula $\scal f \varphi =\int_{\R^d} f(t) \varphi(t) \dd t$, for all $\varphi \in \Scd$.
\end{rem}

We will also consider the space $\mathcal D(\R^d)$ of smooth compactly supported functions and its topological dual $\mathcal D'(\R^d)$, the space of distributions (see for example \cite{distributions_schwartz} or \cite{treves}). 

\section{L\'evy processes and L\'evy noise in $\tem$}\label{levy_dim_one}

\indent\indent 
 A L\'evy process $(X_t)_{t\in \R_+}$ is a real valued stochastic process such that $X_0=0$ almost surely, $X$ has stationary and independent increments and $X$ is stochastically continuous (that is, for any $s\s 0, \ |X_t-X_s| \to 0$ in probability as $t\to s$). Every L\'evy process has a \textit{càdlàg} (right continuous with left limits) modification by  \cite[Theorem 11.5]{sato}, and we will always consider such a modification in the following. An important feature of L\'evy processes is the L\'evy-Itô decomposition: for a L\'evy process $X$ there exists a unique triplet $(\gamma, \sigma, \nu)$, where $\sigma \s 0$, $\gamma\in \R$, and $\nu$ is a L\'evy measure (in particular, $\nu$ is nonnegative and $\int_{\R\backslash \{0\}} \lp 1\wedge |x|^2 \rp\nu( \! \dd x)<+\infty$), such that the jump measure of $X$ (denoted by $J_X$) is a Poisson random measure on $\R_+\times \R\backslash\{0\}$ with intensity $\dd t \, \nu( \! \dd x)$ and $X$ has the decomposition $X_t=\gamma t +\sigma W_t + X^P_t+ X^M_t$. In this decomposition, $W$ is a standard Brownian motion, $X^P_t= \int_{s\in [0,t], \, |x|>1}x J_X ( \! \dd s, \dd x)$ is a compound Poisson process (the term containing the large jumps of $X$), and $X^M_t= \int_{s\in [0,t], \, |x|\I1}x \lp J_X ( \! \dd s, \dd x)-\dd s \nu( \! \dd x) \rp$ is a square integrable martingale (the term containing the small jumps of $X$).

Since a L\'evy process is \emph{càdlàg}, it is locally Lebesgue integrable, and defines almost surely an element of $\m D'(\R)$ via the $L^2$-inner product
$$\scal{X}{\varphi}=\int_{\R_+} X_t \varphi (t) \dd t\, , \qquad \text{for all} \ \varphi \in \m D(\R) .$$
For any \emph{càdlàg} process $L$, we define the following subset of $\Omega$:
\begin{align}\label{omega_set}
  \Omega_{L}= \left \{ \omega \in \Omega:  L(\omega)\in \tem \right \}\, ,
\end{align}
with the understanding that when $L(\omega) \in \tem$, the continuous linear functional associated with $L(\omega)$ is given by $\scal{L(\omega)}{\varphi}=\int_{\R_+} L_t(\omega)\varphi(t) \dd t$, for all $\varphi \in \Sc$.

\subsection{The case of an integrable L\'evy process}
\indent\indent
In order to handle the first three terms in the L\'evy-Itô decomposition, we first consider the case where $X$ itself has a finite first absolute moment.
\begin{prop}\label{martingale}
Let $X$ be a L\'evy process for which $\E\lp |X_1| \rp <+\infty$. Then the set $\Omega_X$ defined in \eqref{omega_set} (with 
$L$ there replaced by $X$) has probability one.
\end{prop}
\begin{proof}
By the strong law of large numbers for L\'evy processes such that $\E\lp |X_1| \rp <+\infty$ in \cite[Theorem 36.5]{sato}, $t^{-1}|X_t| \to \E(X_1)$ almost surely as $t\to +\infty$. It follows that $X$ is sublinear and locally bounded (by the \textit{càdlàg} property) almost surely, so it is slowly growing, which concludes the proof by Remark \ref{rkpolgrowth}.
\end{proof}

\begin{cor} \label{martingalecor}
 Let $X$ be a L\'evy process with characteristic triplet $(\gamma, \sigma, \nu)$ and L\'evy-Itô decomposition $X_t=\gamma t + \sigma W_t +X^P_t+ X^M_t$. Let $Y_t=\gamma t +\sigma W_t+X^M_t$. Then $Y$ is slowly growing a.s., and the set $\Omega_Y$ defined as in \eqref{omega_set} has probability one.
 \end{cor}
\begin{proof}
The process $\tilde Y=\sigma W+X^M$ is a sum of two independent square integrable L\'evy processes with mean zero. Hence $\tilde Y$ verifies the hypothesis of the Proposition \ref{martingale}, therefore it defines a tempered distribution a.s. Since $\tilde Y$ and $Y$ differ by a slowly growing function $t\mapsto \gamma t$, we deduce that $Y$ is a tempered distribution almost surely.  
\end{proof}

\subsection{Growth of a compound Poisson process}\label{sectionpoisson}

\indent\indent
In view of Corollary \ref{martingalecor}, it remains to determine when a compound Poisson process belongs to $\tem$. We begin with two key results on the growth of a compound Poisson process. Let $X_t=\sum_{i=1}^{N_t} Y_i$ be a compound Poisson process, where $N$ is a Poisson process with parameter $\lambda$ that is independent of the sequence $(Y_i)_{i\s 1}$ of i.i.d.~random variables. Let $S_0=0$ and $(S_n)_{n\s 1}$ be the sequence of jump times of $X$ and let $T_n=S_n-S_{n-1}$. Also, let $Z_n=X_{S_n}=\sum_{i=1}^n Y_i$. We first show that on the set $\Omega_X$, the compound Poisson process is slowly growing.
\begin{prop}\label{growth}
 Let $X$ be the compound Poisson process defined above and $\Omega_X$ the set defined in \eqref{omega_set}. There is a set $A$ of probability one such that for all $\omega\in \Omega_X \cap A$, the function $t\mapsto X_t(\omega)$ is slowly growing.
\end{prop}

\begin{rem}
 We point out that this result relies on more than the piecewise constancy of a compound Poisson process. Indeed, there exists \emph{càdlàg} piecewise constant functions in $\tem$ which are not slowly growing. For example consider the function $f$ that is equal to zero except on intervals of the form $[n, n+2^{-n}[$ where it is constant equal to $2^{\frac n 2}$ for all $n\in \N$. Then $f\in L^1(\R) \subset \tem$, but $f$ is clearly not slowly growing.
\end{rem}

\begin{proof}[Proof of Proposition \ref{growth}]
The main idea is the following. Since $X$ is constant on the interval $[S_n , S_{n+1}[$ and the jump times are rarely close together, we can build a sequence of random test functions $\varphi_n$ supported just to the right of $S_n$, such that $\scal X {\varphi_n}= X_{S_n}$ for large enough $n$. The control of $\left | \scal X {\varphi_n} \right |$ by a norm $\m N_p(\varphi_n)$ leads to a bound on $X_{S_n}$, and then on $X_t$ since $X$ is piecewise constant. 

For $n\s 1$, the jump time $S_n$ has Gamma distribution with parameters $n$ and $\lambda$. For $k\s 1$ to be chosen later, and $\varphi \in \mathcal D(\R)$ with support in $[0, 1]$, $\varphi \s 0$ and $\int_\R \varphi=1$, we consider the sequence $\varphi_n$ defined by 
\begin{equation}\label{phin}
 \varphi_n(t)= S_n^k\varphi\lp \lp t-S_n\rp S_n^k \rp \, .
\end{equation}
Then
\begin{equation}\label{support}
 \text{supp} \,(\varphi_n) \subset  \left [S_n, S_n+\frac{1}{S_n^k} \right] \, ,
\end{equation}
 and $\int_\R \varphi_n=1$. Furthermore, for any nonnegative integer $p$ and $\alpha, \beta \I p$,
 \begin{align*}
 \sup_{x\in \R} \left |x^\alpha  \varphi^{(\beta)}_n(x) \right | &= \sup_{x\in \left [S_n, S_n+ \frac 1 {S_n^k} \right ]} \left |x^\alpha \varphi^{(\beta)}_n(x) \right | 
 \I \lp S_n+ \frac 1 {S_n^k} \rp^\alpha S_n^{k(\beta+1)} \sup_{x\in \R} \left | \varphi^{(\beta)}(x) \right | \, ,
\end{align*}
hence,
 \begin{align}\label{np2}
\mathcal N_p\lp \varphi_n \rp \mathds 1_{S_n\s 1} \I  C \mathcal N_p(\varphi) S_n^{(p+1)k+p} \mathds 1_{S_n\s 1} \, ,
\end{align}
where $C\in \R$ is deterministic, nonnegative, and depends only on $p$. We define the events
\begin{equation}\label{setA}
   A_{n,k}=\left\{ X \ \text{does not jump in the interval} \ \left ] S_n, S_n+\frac{1}{S_n^k} \right [ \, \right\} \, .
\end{equation}
Using the fact that $T_{n+1}$ has exponential distribution with parameter $\lambda$ and that $T_{n+1}$ and $S_n$ are independent, we have
\begin{align*}
 \PR(A_{n,k}^c) &=\PR\left \{ N_{S_n+\frac{1}{S_n^k}}-N_{S_n} \s 1 \right \} 
 =\PR\left \{ T_{n+1}< \frac{1}{S_n^k} \right \}
 =\E\lp1-e^{-\frac{\lambda}{S_n^k}}\rp  \I \E\lp \frac {\lambda }{S_n^k} \rp \, .
\end{align*}
The Laplace transform of $S_n$ is $ \E\lp e^{-tS_n}\rp= \lambda^n(t+\lambda)^{-n}$, for $t\s 0$. For $n\s 3$, integrating twice from $t$ to $+\infty$, we obtain
\begin{equation}\label{gammaint}
 \E\lp \frac 1 {S_n^2} \rp= \frac{\lambda^2}{(n-1)(n-2)} \, , \qquad n\s 3 \, .
\end{equation}
 We deduce that $\sum_n \E\lp   S_n^{-2} \rp <+\infty $.
Taking $k=2$, we deduce that $\sum_n \PR\lp A_{n,2}^c \rp <+\infty$ and by the Borel-Cantelli Lemma,
$$\PR \lp \underset{n\to +\infty}{\lim\sup} \ A^c_{n,2}\rp=0 \, ,$$
and the set $A=\underset{n\to +\infty}{\lim\inf} A_{n,2}$ has probability one. Let $\omega \in A \cap \Omega_X$, and $N(\omega)$ be such that for all $n\s N(\omega), \ \omega\in A_{n,2}$. Then for $n\s N(\omega)$, because of \eqref{support} and \eqref{setA},
\begin{align}\label{eq2}
 \langle X, \varphi_n\rangle(\omega)= X_{S_n}(\omega)\mathds 1_{A_{n,2}}(\omega)+ \scal X {\varphi_n} (\omega) \mathds 1_{A_{n,2}^c}(\omega)=X_{S_n}(\omega) \, .
\end{align}
 Since $X(\omega)$ is a tempered distribution by definition of $\Omega_X$, there is $p(\omega)\in \N$ and $C(\omega)\in \R$ such that 
\begin{align}\label{estimation}
\left | \langle X, \varphi_n\rangle (\omega) \right| \mathds 1_{S_n(\omega) \s 1} &\I C(\omega) \mathcal N_{p(\omega)}\lp \varphi_n \rp \mathds 1_{S_n(\omega)\s 1} 
\I C'(\omega) S_n^{3p(\omega)+2}(\omega) \mathds 1_{S_n(\omega)\s 1}
\end{align}
 by \eqref{np2} with $k=2$. Because $S_n\to {+\infty}$ a.s., we can choose $N(\omega)$ such that $S_n(\omega) \s 1$ for all integers $n\s N(\omega)$ (replacing $A$ by another almost sure set). From \eqref{eq2} and \eqref{estimation}, we deduce that for all $\omega \in A\cap \Omega_X$, 
$$ \frac{\left | X_{S_n}(\omega) \right |}{S_n^{3p(\omega)+2}(\omega)} \I C'(\omega)<+\infty \, , \qquad \text{for all}\  n \s N(\omega)\, . $$
 Let $n\s N(\omega)$ and let $t\s S_n(\omega)$. There is an integer $j\s n$ such that $t\in \left [S_j(\omega), S_{j+1}(\omega)\right [\,$. Then
\begin{align*}
 |X_t (\omega)| =|X_{S_j} (\omega)| \I C'(\omega) S_j^{3p(\omega)+2}(\omega) \I C'(\omega) t^{3p(\omega)+2} \, .
\end{align*}
We deduce that
$$\underset{t\to +\infty}{\lim\sup} \, \frac{|X_t (\omega)|}{1+ t^{3p(\omega)+2}} \I C'(\omega)<+\infty$$
on the set $A \cap \Omega_X$. This completes the proof.
\end{proof}

%\noindent
The next proposition recalls properties of the long term behavior of a compound Poisson process. Similar results on the growth of L\'evy processes are available in \cite[Proposition 48.10]{sato}. We include a proof for convenience of the reader.
\begin{prop}\label{longterm}
  Let $X$ be the compound Poisson process with jump heights $(Y_i)_{i\s 1}$ defined at the beginning of this section. 
\begin{itemize}
\item[(i)] Suppose that there is a real number $p>0$ such that $\E(|Y_1|^p)<+\infty$. Then there is $\alpha > 0$ such that 
$$\underset{t\to +\infty}{\lim\sup} \frac {|X_t|}{1+t^\alpha} <+ \infty \qquad \text{a.s.}$$
 \item[(ii)] Suppose that $\E(|Y_1|^p)=+\infty$ for every $p>0$. Then for any $\alpha > 0$,
$$\underset{t\to +\infty}{\lim \sup} \frac {|X_t|}{1+t^\alpha} = +\infty \qquad \text{a.s.}$$ 
\end{itemize}
\end{prop}

\begin{proof}
We use the notations introduced at the beginning of Section \ref{sectionpoisson}. To prove \textit{(i)}, let $p>0$ be such that $\E(|Y_1|^p)<+\infty$. If $p<1$, then by the law of large numbers of Kolmogorov, Marcinkiewicz and Zygmund (see \cite[Theorem 4.23]{kallenberg}), we have $n^{-\alpha}Z_n\to  0$ a.s., with $\alpha= p^{-1}$, so $\sup_{n\s 1} n^{-\alpha}|Z_n| < +\infty$ a.s.
 If $p\s 1$, then by the strong law of large numbers, $\sup_{n\s 1} n^{-1} Z_n < + \infty$. Finally, for $p>0$, we combine both cases by setting $\alpha=\max \lp p^{-1} , 1 \rp$, so that
\begin{align}\label{sup}
 \sup_{n\s 1} \frac {|Z_n|}{1+n^\alpha} <+ \infty \qquad \text{a.s.}
\end{align}

   Let $t\in \R_+$. There is an integer $k$ such that $t\in [S_k, S_{k+1}[$, so that $X_t=X_{S_k}=Z_k$ and
\begin{align}\label{supi1}
 \frac{|X_t|}{1+t^\alpha} \I \frac{|Z_k|}{1+S_k^\alpha} = \frac{|Z_k|}{1+k^\alpha}\frac{1+k^\alpha}{1+S_k^\alpha}\, .
\end{align}
 Since $S_k$ is the sum of $k$ i.i.d.~exponential random variables with parameter $\lambda>0$, the law of large numbers tells us that $k^{-1}S_k\to \frac 1 \lambda$ a.s. We deduce from \eqref{supi1} and \eqref{sup} that  
\begin{align}\label{limsup}
 \underset{t\to +\infty}{\lim\sup} \, \frac {|X_t|}{1+t^\alpha} < +\infty \qquad \text{a.s.},
\end{align}
and \textit{(i)} is proved.

To prove \textit{(ii)}, suppose that for any $p>0$, we have $\E(|Y_1|^p)=+\infty$. Then according to the theorem in \cite{kallenberg} mentioned above, for any $p \in \, ]0,1[, \ n^{-1/p} Z_n$ does not converge on a set of positive probability. Since $\lp Z_n\rp_{n\s 1}$ is a sum of i.i.d.~random variables, the existence of a limit at infinity is a tail event. From Kolmogorov's zero-one law, we deduce that for any $p \in \, ]0,1[ , \  n^{-1/p} Z_n$ does not converge almost surely, and, in particular,
\begin{align}\label{limsup2}
  \limsup_{n\to +\infty} \frac{|Z_n|}{n^{1/p}}>0 \qquad \text{a.s.} \, .
\end{align}
 Fix $\alpha>0$ and let $p_1=\frac 1 {\alpha+1} \in \, ]0,1[ $. By \eqref{limsup2}, 
\begin{align}\label{supi2}
\limsup_{n\to +\infty} \frac{|Z_n|}{n^{1/{p_1}}}>0 \qquad \text{a.s.}
\end{align}
 For $t\in \R_+$, there is an integer $k$ such that $t\in [S_k, S_{k+1}[$, so $X_t=X_{S_k}=Z_k$ and
\begin{align}\label{supi3}
 \frac{|X_t|}{1+t^\alpha} \s \frac{|Z_k|}{1+S_{k+1}^\alpha} = \frac{|Z_k|}{1+k^{1/{p_1}}}\frac{1+k^{1/{p_1}}}{1+S_{k+1}^\alpha} \, .
\end{align}
By the strong law of large numbers and the fact that ${p_1}^{-1}=\alpha+1>\alpha$, we have that $\lim_{k\to \infty} \frac{1+k^{1/{p_1}}}{1+S_{k+1}^\alpha}=+\infty$. Taking the $\limsup$ on both sides of \eqref{supi3} (in fact taking the limit along some subsequence), we deduce from \eqref{supi2} that
 $$ \underset{t\to +\infty}{\lim\sup} \, \frac {|X_t|}{1+t^\alpha} =+ \infty  \qquad \text{a.s.}$$
\end{proof}

\begin{cor}\label{levyprocesstempered}
 Let $X$ be a L\'evy process with characteristic triplet $(\gamma, \sigma ,\nu)$ and $\Omega_X$ be the set defined as in \eqref{omega_set}. 
 \begin{itemize}
 \item[(i)] If there exists $\eta>0$ such that $\E\lp |X_1|^\eta\rp<+\infty$, then $\PR(\Omega_X)=1$.
 \item[(ii)] If for all $\eta>0, \ \E\lp |X_1|^\eta\rp=+\infty$, then $\PR(\Omega_X)=0$.
 \end{itemize}
\end{cor}

\begin{rem}\label{equiv}
If $\E\lp |X_1|^\eta \rp <+\infty$ for some $\eta >0$, then we say that $X$ has a positive absolute moment (\textbf{PAM}). Recall that for $\eta >0$, $\E\lp |X_1|^\eta \rp <+\infty$ if and only if $\int_{|x|>1}|x|^\eta \nu( \! \dd x)<+\infty$ (see \cite[Theorem 25.3]{sato}). Hence the condition \pam \ can be equivalently expressed in terms of the L\'evy measure $\nu$.
\end{rem}

\begin{proof}[Proof of Corollary \ref{levyprocesstempered}]
 To prove \textit{(i)}, let $X_t=\gamma t + \sigma W_t+ X^M_t+ X^P_t$ be the L\'evy-Itô decomposition of $X$. Since $\E\lp |X_1|^\eta\rp<+\infty$, we have $\int_{|x|>1}|x|^\eta \nu( \! \dd x)<+\infty$ (see Remark \ref{equiv}). The jump heights $(Y_i)_{i\s 1}$ of the compound Poisson part $X^P$ are i.i.d., with law $\lambda^{-1}\mathds 1_{|x|>1} \nu(\! \dd x)$ (where $\lambda$ is a normalizing constant), therefore $\E\lp |Y_1|^\eta\rp<+\infty$. Then we can use Proposition \ref{martingale} for the continuous and small jumps terms of the L\'evy-Itô decomposition of $X$, and Proposition \ref{longterm} for the large jumps term, to deduce that $X$ has polynomial growth at infinity. By the \textit{càdlàg} property of $X$ and Remark \ref{rkpolgrowth} we get the result.
 
 To prove \textit{(ii)}, since 
 $$\{\omega : t\mapsto X_t(\omega) \ \text{is slowly growing} \} \cap \{ \omega : \forall \alpha >0\, , \ \limsup_{t\to +\infty} (1+t^\alpha)^{-1} |X_t|=+\infty \}=\emptyset \, ,$$
 and under \textit{(ii)} the second set has probability one by Proposition \ref{longterm}, we deduce from Proposition \ref{growth} that $\PR\lp \Omega_{ X}\cap A \rp =0$, where $A$ is the almost-sure set defined in Proposition \ref{growth}. Therefore, $\PR \lp \Omega_X \rp=0$.
\end{proof}

\subsection{L\'evy white noise: the general case}\label{sec2.3}

\indent\indent
Let $X$ be a L\'evy process. We can define the derivative of $X$ in the sense of distributions as follows.
\begin{df}\label{lnoiseone}
Let $X$ be a L\'evy process with characteristic triplet $(\gamma, \sigma, \nu)$. The L\'evy white noise $\dot X$ is the derivative of $ X$ in $\m D'(\R)$: for $\omega \in \Omega$ and $\varphi \in \m D(\R)$,
$$\scal{\dot X (\omega)}{\varphi}:=-\scal{ X(\omega)}{\varphi '}:=-\int_{\R_+} X_t(\omega) \varphi'(t) \dd t\, .$$
\end{df}\label{defwn}

Notice that the law of the L\'evy white noise $\dot X$ is entirely characterized by the triplet $(\gamma, \sigma, \nu)$ (given that we use the truncation function $\mathds 1_{|x|\I 1}$ in the L\'evy-Itô decomposition). 

\begin{rem}
 Our definition of L\'evy white noise is equivalent to other definitions such as the one found in \cite[Definition 5.4.1]{spdeholden} and in \cite[Definition 3]{fageot}. We postpone the discussion of this issue to the multiparameter case: see Proposition \ref{lnoise}.
\end{rem}

We now turn to the question of whether or not a L\'evy white noise is a tempered distribution. Similar to \eqref{omega_set}, for any L\'evy noise $\dot X$, we define the set  
\begin{align}\label{omega_set_dot}
 \Omega_{\dot X}= \left \{ \omega \in \Omega : \dot X(\omega) \in \tem \right \} \, ,
\end{align}
 and we have the following characterization. 
\begin{theo}\label{tempered}
 Let $X$ be a L\'evy process with characteristic triplet $(\gamma, \sigma, \nu)$, and $\dot X$ the associated L\'evy white noise. Then the following holds for the set $\Omega_{\dot X}$ defined in \eqref{omega_set_dot}:
\begin{itemize}
 \item[(i)] If there exists $\eta>0$ such that $\E\lp |X_1|^\eta\rp<+\infty$, then $\PR\lp \Omega_{\dot X} \rp=1$.
 \item[(ii)] If $\E\lp |X_1|^\eta\rp=+\infty$ for all $\eta>0$, then $\PR\lp \Omega_{\dot X} \rp=0$.
\end{itemize}
\end{theo}

\begin{proof}
  Suppose that $X$ has a \pam \ of order $\eta$. By Corollary \ref{levyprocesstempered} \textit{(i)}, $P(\Omega_X)=1$.
%The ``large jumps'' part $X^P$ in the L\'evy-Itô decomposition of $X$ is a compound Poisson process with intensity $\lambda= \int_{|x|>1} \nu( \! \dd x)$ and the law of its jumps $(Y_i)_{i\s 1}$ is given by $\mu( \! \dd x)=\lambda^{-1}\mathds 1_{|x|>1}\nu( \! \dd x)$. By Remark \ref{equiv} and Corollary \ref{levyprocesstempered} \textit{(i)}, $\PR\lp \Omega_{ X^P} \rp=1$. By Corollary \ref{martingalecor}, the set $\Omega_{X-X^P}$ has probability one. 
Differentiation maps $\tem$ to itself, hence on $\Omega_X$, the L\'evy noise $\dot X$ is a tempered distribution: $\Omega_{\dot X} \supset \Omega_X$.
%$ \supset \lp \Omega_{X^P}\cap \Omega_{X-X^P}\rp$, 
We deduce that $\Omega_{\dot X}$ has probability one.

Suppose that for all $\eta >0$, $\E\lp |X_1|^\eta\rp=+\infty$. Let $X^P_t=\sum_{i=1}^{N_t} Y_i$ be the compound Poisson part of the decomposition of $X$, then by Remark \ref{equiv}, and from the fact that $(Y_i)_{i\s 1}$ is a sequence of i.i.d.~random variables with law $\mathds 1_{|x|>1} \nu(\dd x)$, we deduce that $\E\lp |Y_1|^\eta\rp=+\infty$ for all $\eta>0$. By Corollary \ref{martingalecor}, $\Omega_X=\Omega_{X^P}$, and by Corollary \ref{levyprocesstempered} \textit{(ii)}, $\PR(\Omega_{X})=\PR(\Omega_{X^P})=0$. We now show that $\Omega_{\dot X} \subset \Omega_X$. Let $\omega\in \Omega_{\dot X}$. Two solutions in $\m D'(\R)$ of the equation $u'=\dot X(\omega)$ differ by a constant (see \cite[Théorème I, chapter II, \S 4 p.51]{distributions_schwartz}) and $X(\omega)$ is obviously one of them. Therefore, if there is a solution to this equation in $\tem$, then $\omega\in \Omega_X$. To show that such a solution $u$ exists, recall that a distribution is an element of $\tem$ if and only if it is the derivative of some order of a slowly growing continuous function (see \cite[Théorème VI, chapter VII, \S 4 p.239]{distributions_schwartz}): $\dot X(\omega)=g^{(n)}$ for some continuous slowly growing function $g$ and some integer $n$. If $n\s1$, then $u=g^{(n-1)}$ is a solution in $\tem$ of $u'=\dot X(\omega)$. If $n=0$, then $u(t)=\int_0^t g(s) \dd s$ is a slowly growing solution, therefore $u\in \tem$.
\end{proof}

\begin{cor}\label{measurable}
 Let $X$ be a L\'evy process with characteristic triplet $(\gamma, \sigma, \nu)$, let $\dot X$ be the associated L\'evy noise and suppose it has a \textbf{PAM}. Then there is a random tempered distribution $S$, that is a measurable map from $\lp \Omega, \mathcal F\rp$ to $\lp \mathcal S'(\R), \mathcal B\rp$, where $\mathcal B$ is the Borel $\sigma$-field for the weak-$*$ topology, such that almost surely, for all $\varphi \in \Sc$,
 $$\scal{S}{\varphi}= \scal{\dot X}{\varphi}=-\int_{\R_+} X_t \varphi'(t) \dd t\, .$$
  In addition, the maps $C: \omega \mapsto C(\omega)$ and $p: \omega \mapsto p(\omega)$ such that for all $\varphi \in \Sc$,
 $$ \left |\scal{S}{\varphi}\right | \I C \mathcal N_p(\varphi) \qquad \text{a.s.}$$
 can be chosen to be $\mathcal F$-measurable.
\end{cor}

\begin{proof}
 We already know from Theorem \ref{tempered} that $\PR\lp \Omega_{\dot X} \rp=1$. We define $S$ to be equal to $\dot X$ (in $\tem$) on $\Omega_{\dot X}$ and zero elsewhere. We want to be able to consider $S$ as a measurable map with values in $\mathcal S'(\R)$. More precisely, we equip $\mathcal S'(\R)$ with the weak-$*$ topology. A basis for this topology is given by cylinder sets of the form
 $$O=\bigcap_{i=1}^n\left\{ u\in \mathcal S'(\R) : \scal{u}{\varphi_i} \in A_i\right\},$$
 where, for all $i\I n, \ \varphi_i$ is an element of $\mathcal S(\R)$, $n$ is an integer and $A_i$ is an open set in $\R$. The $\sigma$-field  generated by all cylinder sets is called the cylinder $\sigma$-algebra and is denoted by $\mathcal C$. We first show that $S : \lp \Omega, \mathcal F\rp \longrightarrow \lp \tem, \mathcal C \rp$ is measurable. For this, clearly, it suffices to show that for all cylinder sets $O$ as above, the set $ S^{-1}\lp O\rp =\left \{ \omega \in \Omega : S(\omega) \in O \right \}$ belongs to $\mathcal F$. Clearly,
\begin{align*}
 S^{-1}\lp O\rp &= \bigcap_{i=1}^n \left \{ \omega \in \Omega : \scal{S(\omega)}{ \varphi_i} \in A_{i} \right\} \, .
\end{align*}
The map $(t,\omega) \rightarrow X_t(\omega)$ is jointly measurable so by Fubini's Theorem, the map $\scal{S}{ \varphi_i}: \Omega \longrightarrow \R$ is $\mathcal F$-measurable and therefore $ S^{-1}\lp O\rp\in \mathcal F$. The Borel $\sigma$-field $\m B$ contains $\m C$ since every cylinder set is an open set. The converse inclusion is not immediate: see \cite[Proposition 2.1]{minlos} for a proof of the equality $\m B=\m C$. This fact is also mentioned in \cite[ p.41]{fernique}. 

The space $\mathcal S(\R)$ is separable (see \cite[10.3.4 p.176]{pietsch}) and we let $A$ be a countable dense subset. Then the measurability of the maps $C$ and $p$ comes from the fact that we can choose
$$p(\omega)= \min \left\{ p\in \N : \sup_{\varphi \in A} \frac{\left | \scal{S}{\varphi}\right |}{\mathcal N_p (\varphi)}(\omega) <+\infty \right \}\, ,$$
and 
$$C(\omega)= \sup_{\varphi \in A} \frac{\left | \scal{S}{\varphi}\right |}{\mathcal N_{p(\omega)} (\varphi)}(\omega)\, .$$
 \end{proof}

\begin{rem}\label{functional}
 An alternate proof of Theorem \ref{tempered}\textit{(ii)} is as follows. We can restrict to the case where $X$ is a compound Poisson process with jump times $(S_n)_{n\s 1}$. We construct here a solution to the equation $u'(\omega)=\dot X(\omega)$ such that $u(\omega)\in \tem$. Let $\theta \in \mathcal D(\R)$ be such that $\theta\s 0, \ \int_\R\theta=1$ and $\text{supp}\, \theta \subset [0,1]$. Then let $\varphi\in \Sc$. There exists a function $\Phi \in \Sc$ such that $\varphi =\Phi'$ if and only if $\int_\R \varphi=0$ (consider $\Phi(x)=\int_{-\infty}^x \varphi(t) dt$ for the if part, the other direction is obvious). Then consider the linear functional $I$ on $\Sc$ defined by 
\begin{align}\label{iphi}
 I \varphi(t)=\int_{-\infty}^t \lp \varphi(s)-\theta(s) \int_\R \varphi\rp \dd s\, .
\end{align}
This functional defines an antiderivative on $\Sc$: for any $\varphi \in \Sc, \, I\lp \varphi ' \rp= \varphi$. Also, the reader can easily check that for all $p \in \N$, 
$$\sup_{t\in \R} |t|^p |I\varphi(t)| \I C_p \m N_{p+2}(\varphi)\, ,$$
for some constant $C_p$ depending only on $p$, and therefore, $I$ is a continuous linear functional with values in $\Sc$. 
 
This implies that for $\omega \in \Omega_{\dot X}$, we can define a tempered distribution $u(\omega)$ by
 $$ \scal{u(\omega)}{\varphi}=- \scal{\dot X(\omega)}{I\varphi} \, ,  \qquad \text{for all} \ \varphi \in \Sc\, .$$
This tempered distribution satisfies $u'(\omega)=\dot X(\omega)$, since for any $\varphi \in \Sc$,
$$ \scal{u'(\omega)}{\varphi}= -\scal{u(\omega)}{\varphi '} =\scal{\dot X(\omega)}{I\varphi '}= \scal{\dot X(\omega)}{\varphi}\, .$$
This implies that $u$ and $X$ only differ by a (random) constant. Indeed, 
$$\scal{ u(\omega) }{\varphi}=-\scal{\dot X(\omega)}{I\varphi}=\scal{X(\omega)}{\lp I\varphi\rp^\udot}=\scal{X(\omega)}{\varphi}-\scal{X(\omega)}{\theta}\scal 1 \varphi \, .$$
Therefore, this (random) constant is $\scal{X(\omega)}\theta$, and
\begin{align}\label{xomega}
 X(\omega)=u(\omega)+\scal{X(\omega)}\theta \cdot 1\, ,
\end{align}
 and so $X(\omega)\in \tem$ since the right-hand side belongs to $\tem$. Therefore $\Omega_{\dot X} \subset \Omega_X$. By Corollary \ref{levyprocesstempered} \textit{(ii)}, we conclude that $\PR \lp \Omega_{\dot X}\rp=0$.
\end{rem}

 \section{L\'evy fields and L\'evy noise in $\temd$}\label{levy_dim_d}
 
 \indent\indent

In this section, we consider the same questions as in Section \ref{levy_dim_one}, but for a generalization of the notion of L\'evy process, where the ``time'' parameter is in $\R_+^d$, with $d\s 1$. A general presentation of this theory of \emph{multiparameter L\'evy fields} can be found in \cite{adler}; see also \cite{walshdalang}. 

In the following, for any $k \in \N, \ \mathbf 1_k$ (respectively $\mathbf 0_k, \, \mathbf 2_k$) denotes the $k$-dimensional vector with coordinates all equal to $1$ (respectively to $0,\, 2$). We recall that $(\Omega, \mathcal F , \PR )$ is a complete probability space. Let $(X_t)_{t\in \R_+^d}$ be a $d$-parameter random field. For $s,t \in \R^d_+$ with $s=(s_1,\dots,s_d)$, $t=(t_1,\dots,t_d)$, we say that $s\I t$ if $s_i \I t_i$ for all $1\I i \I d$, and $s<t$ if $s_i < t_i$ for all $1\I i \I d$. For $a\I b \in \R^d_+$, we define the box $]a,b]=\left\{t\in \R_+^d : a<t \I b \right \}$, and the increment $ \Delta_a^b X$ of $X$ over the box $]a,b]$ by
\begin{equation}\label{increment}
 \Delta_a^b X=\sum_{\varepsilon \in \{0,1\}^d}(-1)^{|\varepsilon |}X_{c_\varepsilon (a,b)} \ ,
\end{equation}
where for any $\eps \in \{0,1\}^d$, we write $|\eps|=\sum_{i=1}^d \eps_i$ and $c_\eps(a,b) \in \R^d_+$ is defined by $ c_\eps(a,b)_i=a_i \mathds 1 _{\{\eps_i=1\}}+b_i \mathds 1 _{\{\eps_i=0\}}$, for all $1\I i \I d$.
We can check that when $d=1$, then $\Delta_a^bX=X_b-X_a$. In fact, for all $d\s 1$, $ \int_{[a,b]}\varphi^{ (\mathbf 1_d)}(t) \dd t= \Delta_a^b \varphi$. The next definition is a generalization of the \textit{càdlàg} property to processes indexed by $\R_+^d$. We define the relations $\mathcal R=(\mathcal R_1, ..., \mathcal R_d)$, where $\mathcal R_i$ is either $\I$ or $>$, and $a \,\mathcal R \, b$ if and only if $a_i\mathcal R_i b_i$ for all $1\I i \I d$. 
\begin{df}\label{lamp}
 Using the terminology in \cite{adler} and \cite{straf}, we say that $X$ is \emph{lamp} (for limit along monotone paths) if we have the following: $(i)$ For all $2^d$ relations $\mathcal R$, $\displaystyle \lim_{u\to t, \, t\mathcal R u} \ X_u$ exists; $(ii)$ If $\mathcal R=\lp \I, ..., \I \rp$ then $\displaystyle X_t=\lim_{u\to t, \, t\mathcal R u} \ X_u$; and $(iii)$ $X_t=0$ if $t_i=0$ for some $1\I i \I d$.
\end{df}

\noindent
We are now ready to give the definition of a L\'evy field in $\R_+^d$.

\begin{df}\label{levyfield}
$X=(X_t)_{t\in \R_+^d}$ is a $d$-parameter L\'evy field if it has the following properties:
\begin{itemize}
 \item[(i)]$X$ is \textit{lamp} almost surely.
 \item[(ii)]$X$ is continuous in probability.
 \item[(iii)]For any sequence of disjoint boxes $]a_k,b_k], \ 1\I k \I n$, the random variables $\Delta_{a_k}^{b_k} X$ are independent.
 \item[(iv)]Given two boxes $]a,b]$ and $]c,d]$ such that $]a,b]+t=]c,d]$ for some $t\in \R^d$, the increments $\Delta_a^bX$ and $\Delta_c^dX$ are identically distributed. 
\end{itemize}
The jump $\Delta_t X$ of $X$ at time $t$ is defined by $ \displaystyle \Delta_tX=\lim_{u\to t, \, u<t} \Delta_u^t X$.
\end{df}

This definition coincides with the notion of L\'evy process when $d=1$. In addition, for all $t=(t_1,... , t_d)\in \R_+^d$, and for all $1\I i\I d$, the process $X^{i,t}_\cdot=X_{(t_1,...,t_{i-1},\, \cdot \, ,t_{i+1},...,t_d)}$ is a L\'evy process (the notation here means that it is the process in one parameter obtained by fixing all the coordinates of $t$ except the $i$-th).

%Indeed, by items \textit{(ii)} and \textit{(i)} of Definition \ref{levyfield}, it is continuous in probability and \textit{càdlàg}. Then if $a_k=(0,...,0,r_k,0,...,0)$, where $r_k$ is at the $i$-th position and $b_k=(t_1, ..., t_{i-1}, r_{k+1}, t_{i+1},..., t_d)$ for a sequence $0= r_0 < r_1 <... < r_n$, by item \textit{(iii)} of Definition \ref{levyfield}, the increments $\Delta_{a_k}^{b_k} X$ are independent.  Applying \eqref{increment} and using item $(iii)$ of Definition \ref{lamp}, we find that $\Delta_{a_k}^{b_k} X=X^{i,t}_{r_{k+1}}-X^{i,t}_{r_k}$. Hence $X^{i,t}$ is a \textit{càdlàg}, stochastically continuous process with stationary and independent increments, that is, a Lévy process.

The Brownian sheet is an example of such a $d$-parameter L\'evy field. It is the analog in this framework of Brownian motion and further properties of this field are detailed in \cite{dalang_levelset}, \cite{walshdalang_browniansheet}, \cite{multi} or \cite{spdewalsh}. 

For all $t\in \R_+^d, \ X_t$ is an infinitely divisible random variable, and by the L\'evy-Khintchine formula \cite[Chapter 2, Theorem 8.1 p.37]{sato}, there exists real numbers $\gamma_t,\ \sigma_t$ and a L\'evy measure $\nu_t$ such that $ \E\lp e^{i u X_t} \rp = \exp \left [ i u \gamma_t -\frac 1 2 \sigma^2_t u^2 + \int_\R\lp e^{iux}-1-iux \mathds 1_{|x|\I 1} \rp \nu_t ( \! \dd x) \right ]$. The triplet $(\gamma_t, \sigma_t, \nu_t)$ is called the characteristic triplet of $X_t$. Since for all $1\I i\I d$ and $t\in \R_+^d$, the process $X^{i,t}$ defined above is a L\'evy process, we deduce that there exists a triplet $(\gamma, \sigma, \nu)$ where $\gamma, \sigma \in \R$ and $\nu$ is a L\'evy measure such that $(\gamma_t , \sigma_t, \nu_t)=(\gamma, \sigma, \nu)\text{Leb}_d([0,t])$, where $\leb{d}{\! \dd x}$ is $d$-dimensional Lebesgue measure. We call $(\gamma, \sigma, \nu)$ the characteristic triplet of the L\'evy field $X$. We can now state the multidimensional analog of the L\'evy-Itô decomposition, taken from \cite[Theorem 4.6]{adler} particularized to the case of stationary increments (see also \cite{walshdalang}).

\begin{theo}\label{lid}
 Let $X$ be a $d$-parameter L\'evy field with characteristic triplet $(\gamma, \sigma, \nu)$. The following holds: 
\begin{itemize}
 \item[(i)] The jump measure $J_X$ defined on $\R_+^d \times \lp \R\backslash \{0\} \rp $ by $J_X(B)=\# \left \{ (t, \Delta_t X) \in B   \right \}$, for $B$ in the Borel $\sigma$-algebra of $\R_+^d \times \lp \R\backslash \{0\} \rp$, is a Poisson random measure with intensity $ \text{Leb}_d\times \nu$.
 \item[(ii)] For all $t\in \R_+^d$, we have the decomposition
 $$X_t= \gamma \text{Leb}_d ([0,t])+ \sigma W_t + \int_{[0,t]}\int_{|x|>1}xJ_X( \! \dd s, \dd x)+ \int_{[0,t]}\int_{|x|\I 1}x\tilde J_X( \! \dd s, \dd x)\, ,$$
 where $W$ is a Brownian sheet, $\tilde J_X=J_X- \text{Leb}_d \times \nu$ is the compensated jump measure, and the equality holds almost surely. In addition, the terms of the decomposition are independent random fields.
\end{itemize}
\end{theo}

If $X$ is a $d$-parameter L\'evy field, by the \textit{lamp} property of its sample paths, it is locally bounded and defines almost surely an element of $\m D'(\R^d)$ via the $L^2$-inner product. Similarly to the one-dimensional case (see Definition \ref{defwn}), we now define the $d$-dimensional L\'evy white noise.

\begin{df}\label{levy_noise}
Let $X$ be a $d$-parameter L\'evy field with characteristic triplet $(\gamma, \sigma, \nu)$. The L\'evy white noise $\dot X$ is the $d^{\text {th}}$ cross-derivative of $X$ in the sense of Schwartz distributions: for $\omega \in \Omega$ and $\varphi \in \m D(\R^d)$, 
 $$\scal{\dot X} \varphi (\omega):= (-1)^d \scal{X}{\varphi^{(\mathbf 1_d)}} (\omega) :=(-1)^d\int_{\R_+^d}X_t (\omega) \varphi^{(\mathbf 1_d)}(t) \dd t \, ,$$
 where $\varphi^{(\mathbf 1_d)}=\frac{\partial ^d}{\partial t_1 \cdots \partial t_d} \varphi$. 
\end{df}
As in Section \ref{sec2.3}, note that the law of the multidimensional L\'evy white noise $\dot X$ is entirely characterized by the triplet $(\gamma, \sigma, \nu)$ (given that we use the truncation function $\mathds 1_{|x|\I 1}$ in the L\'evy-Itô decomposition). We will show in Proposition \ref{lnoise} that this definition is equivalent to other definitions of L\'evy white noise.

\begin{rem}\label{rem_stochastic_integral}
 Given a $d$-parameter L\'evy field $X$ with characteristic triplet $(\gamma, \sigma, \nu)$ and jump measure $J_X$, for a suitable class of functions $\varphi : \R_+^d \to \R$, we can define the stochastic integral
\begin{equation}\label{stochastic_integral}
\begin{aligned}
  \int_{\R_+^d} \varphi(s) \dd X_s &:= \gamma \int_{\R_+^d} \varphi(s) \dd s+ \sigma \int_{\R_+^d} \varphi(s) \dd W_s \\
 &\hspace{1cm}+\int_{\R_+^d} \int_{|x| >1} x\varphi(s) J_X(\! \dd s , \dd x)+\int_{\R_+^d} \int_{|x| \I1} x\varphi(s) \tilde J_X(\! \dd s , \dd x) \\
 &=\gamma A_1(\varphi)+\sigma A_2(\varphi)+A_3(\varphi)+A_4(\varphi)\, , 
\end{aligned}
\end{equation}
where the first integral is a Lebesgue integral, the second integral is a Wiener integral (see \cite[Chapter 2]{davar}) and the last two integrals are Poisson integrals (see \cite[Lemma 12.13]{kallenberg}) with the space $S= \R_+^d \times \lp \R\backslash \{0\} \rp$.
\end{rem}

The next lemma relates the definition of L\'evy white noise above with the mapping $\varphi \to \int_{\R_+^d} \varphi (s) \dd X_s$.

\begin{lem}\label{fubinicompact}
 Let $X$ be a $d$-parameter L\'evy field with characteristic triplet $(\gamma, \sigma, \nu)$ and jump measure $J_X$. Then, for all $\varphi \in \m D(\R^d)$, 
\begin{equation}
 \scal{\dot X}{\varphi}= \int_{\R_+^d}\varphi(s) \dd X_s\, .
\end{equation}
\end{lem}

\begin{proof}
 Generically, if $\mu$ is a measure on $\R_+^d$ and if $x(t):=\mu([0,t])$, then $\frac{\partial^d}{\partial t_1 \cdots \partial t_d}x=\mu$ in $\m D'(\R^d)$. Indeed, by \eqref{increment}, for any $\varphi \in \m D(\R^d)$, 
\begin{equation}\label{measure}
\begin{aligned}
  \int_{\R_+^d} \varphi(s) \mu(\! \dd s) &= (-1)^d \int_{\R_+^d} \mu(\!\dd s) \int_{\R_+^d}\dd t \, \varphi^{(\mathbf 1_d)}(t) \mathds 1_{t \s s} \\
 &= (-1)^d \int_{\R_+^d}\dd t \, \varphi^{(\mathbf 1_d)}(t) \int_{\R_+^d}\mu(\!\dd s) \mathds 1_{t \s s}= (-1)^d \int_{\R_+^d}  \varphi^{(\mathbf 1_d)}(t) x(t) \dd t\, ,
\end{aligned}
\end{equation}
where the second equality requires a Fubini-type theorem.

Notice that for bounded Borel sets, the set function
$$ B\mapsto \tilde X(B):= \int_{\R_+^d}\mathds 1_B (s) \dd X_s$$
defines an $L^0(\Omega, \m F, \PR)$-valued measure (see e.g. \cite[Theorem 2.6]{walshdalang}), and $X_t= \tilde X([0,t])$ a.s. We shall apply the argument in \eqref{measure} separately to the four integrals in \eqref{stochastic_integral}. For the first integral, the standard Fubini's theorem applies. For the second integral, since $\varphi\in L^2(\R^d)$, it is well defined, and since it has compact support, we use the stochastic Fubini's theorem \cite[Theorem 2.6]{spdewalsh}. For the third integral, let $J_{X^P}(\! \dd s , \dd x) =\mathds 1_{|x|>1} J_{X}(\! \dd s , \dd x)$ be the jump measure of the compound Poisson part $X^P$ of the L\'evy-Itô decomposition of $X$. Then $J_{X^P}=\sum_{i\s 1}\delta_{\tau_i} \delta_{Y_i}$, where $(\tau_i, Y_i)$ are random elements of  $\R_+^d\times \lp \R\backslash \{0\} \rp$, and $A_3(\varphi)= \sum_{i\s 1} Y_i \varphi(\tau_i)$. For a fixed $\varphi$ with compact support, this is a finite sum, so Fubini's theorem applies trivially. For the term $A_4(\varphi)$, the integral is a compensated Poisson integral, and we know that it exists (see \cite[Lemma 12.13]{kallenberg}) if and only if
\begin{equation}\label{small_jumps_condition}
\int_{\R_+^d} \int_{|x|\I1}  \lp |x\varphi(s)|^2 \wedge |x\varphi(s)| \rp \dd s \, \nu(\! \dd x) <+\infty\, .
\end{equation}
Since $\varphi \in L^2(\R^d)$,
\begin{align*}
 \int_{\R_+^d} \int_{|x|\I1}  \lp |x\varphi(s)|^2 \wedge |x\varphi(s)| \rp \dd s \,  \nu(\! \dd x)  \I \| \varphi \|_{L^2}^2 \int_{|x|\I 1}x^2 \nu(\! \dd x)< +\infty \, .
\end{align*}
For $n\in \N$, define
\begin{align*}
 A_{4,n}(\varphi) :&\! \!= \int_{\R_+^d}\int_{\frac 1 {2^{n+1}} < |x| \I \frac 1 {2^n}} x  \varphi (t) \, \tilde J_X( \! \dd x, \dd t)\\
& \! \!=  \int_{\R_+^d}\int_{\frac 1 {2^{n+1}}  < |x| \I \frac 1 {2^n}} x  \varphi (t) \,  J_X( \! \dd x, \dd t) -  \int_{\R_+^d}\int_{\frac 1 {2^{n+1}} < |x| \I \frac 1 {2^n}} x  \varphi (t) \, \nu(\! \dd x) \dd t  \, .
\end{align*}
Then $A_{4,n}(\varphi)$ is a sequence of centered independent random variables (the compensated Poisson integrals are over disjoint sets) in $L^2$ and $\E\lp \lp A_{4,n}^2(\varphi)\rp \rp=\int_{\R_+^d}\varphi(s)^2 \dd s \int_{\frac 1 {2^{n+1}} \I |x| < \frac 1 {2^n}} x^2 \nu( \! \dd x)$.
 Since $\nu$ is a L\'evy measure, we see that $\sum_n \E\lp \lp A_{4,n}^2(\varphi)\rp \rp< \infty$ and by Kolmogorov's convergence criterion (see \cite[Theorem 2.5.3]{durrett}) we deduce that
 \begin{equation}\label{e3}
  \sum_{0\I k\I n} A_{4,k}(\varphi)  \to  \int_{\R_+^d}\int_{|x| \I1} x  \varphi (s) \, \tilde J_X( \! \dd x, \dd s)=A_4(\varphi) \qquad \text{as} \ n\to +\infty, \ \text{a.s. }
\end{equation}
For each $n\in \N$, since the L\'evy measure $\nu$ is finite on compact subsets of $\R_+^d \times \left [ \frac 1 {2^{n+1}} , \frac 1 {2^n} \right ]$, Fubini's theorem applies to the set function $B\mapsto A_{4,n}(\mathds 1_B)$ in the same way it did for $A_3$ and $A_1$. Therefore, letting 
$$ X^{M,n}_t=\int_{\R_+^d}\int_{\frac 1 {2^{n+1}}  < |x| \I \frac 1 {2^n}} x \mathds 1_{t\s s} \tilde J_X( \! \dd s, \dd x)\, ,$$
the argument in \eqref{measure} implies that
$$A_{4,n}(\varphi)=(-1)^d \int_{\R_+^d} \varphi^{(\mathbf 1_d)}(t) X_t^{M,n} \dd t \, .$$
By \cite[Theorem 4.6]{adler} (see also \cite[Theorem 2.3]{walshdalang}), $\sum_{0\I k\I n} X_t^{M,k} \to X^M_t$, where $X^M$ is the small jumps part of $X$, and the convergence is a.s., uniformly on compact subsets of $\R_+^d$. Since $\varphi$ has compact support, \eqref{e3} implies that
$$A_4(\varphi)=(-1)^d \int_{\R_+^d}  \,  \varphi^{(\mathbf 1_d)}(s)  X^{M}_s  \dd s\, .$$
\end{proof}

\subsection{The case of a $p$-integrable martingale ($p>1$)}

\indent\indent
We say that a random field $M$ is a multiparameter martingale with respect to a filtration $\mathds F =\lp \m F_t\rp_{t\in \R_+^d}$ (see \cite[chapter 7, section 2 p.233]{multi}) if $M$ is $\mathds F$-adapted, integrable, and for all $s\I t \in \R^d_+$, then $\E(M_t |\mathcal F_s)=M_s$. We will also need the notion of commuting filtration (see \cite[Chapter 7, section 2, Definition p.233]{multi} ). By \cite[Theorem 2.1.1 in chapter 7]{multi}, to show that $\mathds F$ is commuting, it suffices to show that for any $s,t\in \R_+^d, \ \mathcal F_s$ and $\mathcal F_t$ are conditionally independent given $\mathcal F_{s\wedge t}$, where $(s\wedge t)_i=s_i\wedge t_i$. In particular, if $X$ is a $d$-parameter L\'evy field and $\mathcal F_t$ is the $\sigma$-algebra generated by the family $(X_s)_{s\I t}$, then $\mathds F$ is commuting by the independence of the increments of $X$.

For any \emph{lamp} random field $L$, we consider, similarly to $\eqref{omega_set}$, the event
\begin{equation}\label{omega_setd}
\Omega_L= \left \{ \omega\in \Omega :  L(\omega)\in \temd  \right \}\, ,
\end{equation}
with the understanding that when $L(\omega)\in \temd$, the continuous linear functional associated with $L(\omega)$ is $\scal{L(\omega)}{\varphi}=\int_{\R_+^d} L_t(\omega)\varphi(t) \dd t$, for all $ \varphi \in \Scd$.
\begin{prop}\label{martingaled}
Fix $p>1$ and let $(M_t)_{t\in \R_+^d}$ be a multiparameter martingale with respect to a commuting filtration $\lp \m F_t \rp_{t\in \R_+^d}$, such that for all $t\in \R_+^d$, $\E \lp | M_t|^p \rp= \lp c \leb{d}{[0,t]}\rp^{\frac p 2}$ for some constant $c$. Then the set $\Omega_M$ defined as in \eqref{omega_setd} has probability one.
\end{prop}

\begin{proof}
 Similarly to the one dimensional case, we control the supremum of $|t|^{-\alpha}|M_t|$ as $|t| \to +\infty$, or, equivalently, the supremum of $|s|^{-\alpha}|M_s|$ for $s \in \R_+^d \setminus [0,t]$ as $\min_{i=1,\dots,d}\, t_i \to  +\infty$, and prove that the limit in probability of this supremum, as all the coordinates of $t$ go to $+\infty$, is zero. The proof uses the multidimensional analog of Doob's $L^p$ inequality: Cairoli's Strong $(p,p)$ inequality (see \cite[Chapter 7, Theorem 2.3.2]{multi}). For all $i\in \N\backslash \{0\}$, let $x_i=2^{i-1}$ and $x_0=0$. For $k=\lp k_1,...,k_d\rp \in \N^d$, let $a_k=(x_{k_1},...,x_{k_d})$, and let $b_k=(2^{k_1},...,2^{k_d})$. We fix $k\in \N^d, \, k \neq (0,\dots,0)$. By using successively Jensen's inequality and Cairoli's inequality, for any $\alpha>0$, we have
\begin{align*}
\E \lp \sup_{s\in [a_k,b_k]}\frac{\left | M_s \right |}{|s|^\alpha} \rp \I \frac 1 {|a_k|^\alpha} \E\lp \sup_{s\I b_k} |M_s|^p\rp^{\frac 1 p}
 \I\frac {c_p} {|a_k|^\alpha} \E\lp  |M_{b_k}|^p\rp^{\frac 1 p}
\I  \frac{c_p\sqrt{c \, \text{Leb}_d([0,b_k])}}{|a_k|^\alpha}\, ,
 \end{align*}
 for some constant $c_p$ depending only on $p$ and the dimension $d$, where $|a_k|$ and $|s|$ denote here the Euclidian norm. Since $k_1 \vee \cdots \vee k_d \s 1$, we have $|a_k|\s 2^{k_1\vee \cdots \vee k_d -1}$, hence 
\begin{equation}
\label{doob} 
 \E \lp \sup_{s\in [a_k,b_k]}\frac{\left | M_s \right |}{|s|^\alpha} \rp\I c_p \sqrt c 2^{\frac 1 2\sum\limits_{i=1}^d k_i}2^{-\alpha \lp k_1\vee \cdots \vee k_d -1 \rp }\I c_p \sqrt c 2^{\alpha}2^{-\lp \frac \alpha d -\frac 1 2\rp\sum\limits_{i=1}^d k_i}\, .
\end{equation}
We choose $\alpha =\lfloor \frac d 2 \rfloor +1$. Let $t\in \R_+^d$ be far enough from the origin (we will consider the limit as all the coordinates of $t$ go to $+\infty$), and for all $1\I i\I d$, let $n_i$ be the largest integer such that $2^{n_i}\I t_i$ and let $n=(n_1,...,n_d)$. We can suppose that $n_i \s 2$ for all $1\I i\I n$. We write $\Xi$ for the set of all relations $\mathcal R$ of the form $(r_1,..., r_d)$, where for all $i\in \left \{ 1, ..., d\right \}, \ r_i \in \left \{ \I, \s \right \}$ and $\mathcal R \neq (\I,...,\I)$. Then $\left [0, t_n\right ] \subset [0,t]$, where $t_n=(2^{n_1},...,2^{n_d})$. The complement of the box $[0,t_n]$ in $\R_+^d$ is covered by boxes of the form $[a_k, b_k]$,  where $k\in \N^d$ and $k \m R n$ for some $\m R\in \Xi$. Therefore,
\begin{align*}
 \PR\lp \sup_{s\notin [0,t]} \frac{|M_s|}{|s|^\alpha} > \varepsilon \rp \I \PR\lp \sup_{s\notin [0,t_n]} \frac{|M_s|}{|s|^\alpha} > \varepsilon   \rp
& \I \sum_{\mathcal R \in \Xi}\underset{ k\mathcal R n}{\sum_{k\in  \N^d}} \PR\lp \sup_{s\in [a_k,b_k]} \frac{|M_s|}{|s|^\alpha} > \varepsilon   \rp\\
 &\I \frac {c_p \sqrt c 2^{\alpha}} \varepsilon \sum_{\mathcal R \in \Xi} \underset{ k\mathcal R n}{\sum_{k\in  \N^d}} \ 2^{-\lp \frac \alpha d -\frac 1 2\rp\sum\limits_{i=1}^d k_i} \underset{t  \twoheadrightarrow +\infty}{\longrightarrow} 0 \, ,
\end{align*}
where $t \twoheadrightarrow +\infty$ means that $t_1\wedge ... \wedge t_d \to +\infty$. To check that the limit is indeed zero, one has that for any fixed $\m R \in \Xi$, at least one of the inequalities in $\m R$ is $\geq$. By symmetry, we can suppose that it is the first inequality. Then
$$\underset{ k\mathcal R n}{\sum_{k\in  \N^d}} \ 2^{-\lp \frac \alpha d -\frac 1 2\rp\sum\limits_{i=1}^d k_i} \leq C_{\alpha, d} \sum_{k_1 \s n_1}2^{-\lp \frac \alpha d -\frac 1 2\rp k_1} \underset{n_1\to +\infty} \to 0\, .$$
The result follows since $\Xi$ is a finite set. Then $\sup_{s\notin [0,t]} |s|^{-\alpha}|X_s^M|  \to 0$ in probability as $t  \twoheadrightarrow +\infty$, therefore $|t|^{-\alpha}|M_t| \to 0$ a.s as $|t|  \to +\infty$. 
By the \emph{lamp} property of $M$ we deduce that $M$ is slowly growing, and by Remark \ref{rkpolgrowth} we deduce that $\PR(\Omega_M)=1$.
\end{proof}

\noindent
\begin{cor} \label{martingalecord}
 Let $X$ be a $d$-parameter L\'evy field with characteristic triplet $(\gamma, \sigma, \nu)$ and L\'evy-Itô decomposition $X_t=\gamma \leb{d}{[0,t]} + \sigma W_t +X^P_t+ X^M_t$ where $X^P$ is the large jump part of the decomposition and $X^M$ is the compensated small jumps part. Let $Y_t=\gamma \leb{d}{[0, t]} +\sigma W_t+X^M_t$. Then the set $\Omega_Y$ defined in \eqref{omega_setd} has probability one.
 \end{cor}
\begin{proof}
 The random field $\tilde Y=\sigma W+X^M$ is a sum of two independent square integrable martingales and by a classical result on compensated Poisson integrals and Brownian sheets,
  $$\E \lp \tilde  Y^2_t \rp= \lp \sigma^2 + \int_{|x|\I 1} x^2 \nu( \! \dd x) \rp \leb{d}{[0,t]}\, ,$$  
 where the multiplicative constant is finite since $\nu$ is a L\'evy measure. Hence $\tilde Y$ verifies the hypothesis of the Proposition \ref{martingaled} with $p=2$, therefore it defines a tempered distribution a.s. Since $\tilde Y$ and $Y$ differ by a slowly growing function, we deduce that $Y$ is a tempered distribution almost surely.  
\end{proof}

\subsection{The compound Poisson sheet}

\indent\indent
By Corollary \ref{martingalecord}, for any $d$-parameter L\'evy field $X$, we have $\Omega_X \cap \Omega_Y=\Omega_{X^P}\cap \Omega_Y$.  We shall prove that $\Omega_{X^P}$ has probability $0$ or $1$. In the one dimensional setting, we used the fact that a compound Poisson process with a \pam \ is slowly growing a.s (see Proposition \ref{longterm}\textit{(i)}). As mentioned in the Introduction, the same results in a $d$-dimensional setting are to the best of our knowledge unavailable, which leads us to find another approach. In the multiparameter case, we will use properties of stochastic integrals with respect to a Poisson random measure to show that under a moment condition, a compound Poisson sheet and its associated white noise define tempered distributions. While this is in principle a special case of \cite[Theorem 3]{fageot}, in view of Corollary \ref{martingalecord}, the two statements are in fact equivalent.
 \begin{lem}\label{pnoise}
Let $\nu$ be a L\'evy measure and $M$ be a Poisson random measure on $\lp \R\backslash \{0\} \rp \times \R_+^d $ with intensity measure $ \mathds{1}_{|x|>\eta}\nu( \! \dd x)\dd t$, where $\eta>0$. Suppose that $ \int_{|x|>\eta}|x|^\alpha \nu( \! \dd x)<+\infty$ for some $\alpha>0$ (\pam) and consider the compound Poisson sheet $P_t=\int_{[0,t]}\int_{|x|>\eta}xM( \! \dd t, \dd x)$. Then 
\begin{itemize}
 \item[(i)] $M$ almost surely defines a tempered distribution via the formula
 \end{itemize}
\begin{equation}\label{def1}
  \scal M \varphi =\int_{\R_+^d}\int_{|x|>\eta}M( \! \dd t, \dd x)  \varphi(t)x \, , \qquad \varphi \in \Scd \, .
\end{equation}
\begin{itemize}
 \item[(ii)] $\PR(\Omega_{P})=1$ and for all $\varphi \in \Scd$,
 \end{itemize}
\begin{equation}\label{def2}
  \scal{ P}{\varphi}:=\int_{\R_+^d}P_t \varphi(t) \dd t= \int_{\R_+^d}\int_{|x|>\eta}M( \! \dd t, \dd x) \int_{[t,+\infty[}\dd s \, \varphi(s)x \, , 
\end{equation}
\begin{itemize}
 \item[(iii)]$M=P^{(\mathbf 1_d)}$ in $\temd$, where we recall that $P^{(\mathbf 1_d)}=\frac{\partial^d }{\partial t_1\cdots \partial t_d}P$.
\end{itemize}
\end{lem}

\begin{proof}
Since $M$ is a Poisson random measure on $\R_+^d\times\lp \R\backslash \left \{ 0\right \}\rp  $ with jumps of size larger than $\eta$, there are (random) points $(\tau_i, Y_i)_{i\s 1}\in  \R_+^d \times \lp \R\backslash [-1,1]\rp$ such that $M=\sum_{i\s 1}\delta_{\tau_i} \delta_{Y_i} $. To prove \textit{(i)}, we first need to check that the integral in \eqref{def1} is well defined. Let $\varphi \in \Scd$. The stochastic integral is a Poisson integral, and it is well defined (as the limit in probability of Poisson integrals of elementary functions) if and only if (see \cite[Lemma 12.13]{kallenberg})
\begin{equation}\label{condition}
 \int_{|x|>\eta} \int_{\R_+^d} \lp \left |x\varphi(t)\right | \wedge 1\rp \dd t \, \nu( \! \dd x) <+\infty \, .
\end{equation}
 Let $r\in \N$. There is a constant $C>1$ such that $\sup_{t\in \R^d_+} (1+|t|^r)|\varphi(t)| \I C < +\infty$. Then $ \left |x\varphi(t)\right | \wedge 1 \I \frac{C|x|}{1+|t|^r} \wedge 1$. We write $V_d$ for the volume of the $d$-dimensional unit sphere. Then, for $|x|>1$,
\begin{align*}
 \int_{\R_+^d} \lp \left |x\varphi(t)\right | \wedge 1\rp \dd t & \I \int_{\R_+^d} \lp \frac{C|x|}{1+|t|^r} \wedge 1\rp \dd t \\
 &\I d V_d \int_{\R_+} \lp \frac{C|x|}{1+u^r} \wedge 1\rp u^{d-1} \dd u \\
 & \I d V_d \lp \int_0^{\lp C|x|-1 \rp^{\frac 1 r}} u^{d-1} \dd u+C|x| \int_{\lp C|x|-1 \rp^{\frac 1 r}}^{+\infty} \frac{u^{d-1}}{1+u^r} \dd u \rp \\
 &\I V_d \lp C|x|-1 \rp^{\frac d r}+d V_d C|x| \int_{\lp C|x|-1 \rp^{\frac 1 r}}^{+\infty} \frac{u^{d-1}}{1+u^r} \dd u \, .
\end{align*}
The last integral has to be well defined so we take $r>d$, and then
\begin{align*}
  \int_{\lp C|x|-1 \rp^{\frac 1 r}}^{+\infty} \frac{u^{d-1}}{1+u^r} \dd u \I \int_{\lp C|x|-1 \rp^{\frac 1 r}}^{+\infty} u^{d-1-r} \dd u = \frac 1 {r-d} \lp C|x|-1 \rp^{\frac {d-r} r} \, ,
\end{align*}
so 
$$ \int_{\R_+^d} \lp \left |x\varphi(t)\right | \wedge 1\rp \dd t \I V_d \lp C|x|-1 \rp^{\frac d r}+\frac{d V_d C|x| }{r-d}\lp C|x|-1 \rp^{\frac {d-r} r}\, .$$
 We deduce that there exists a constant $C'$ such that for $|x|>1$,
\begin{equation}\label{eq3.12}
  \int_{\R_+^d} \lp \left |x\varphi(t)\right | \wedge 1\rp \dd t \I C' |x|^{\frac d r} \, .
\end{equation}
We then choose $r$ large enough so that $\frac d r\I \alpha \wedge \frac 1 2$, in which case the moment condition on $\nu$ gives us \eqref{condition}, and therefore the Poisson integral is well defined and a.s. finite. Set $g_r(t)= \frac 1 {1+|t|^r}\, , \ t\in \R^d_+$. Then for $r$ sufficiently large,
$$\int_{\R_+^d}\int_{|x|>\eta} M(\! \dd t , \dd x) g_r(t) |x|$$
is well-defined, since by \eqref{eq3.12} and \pam,
$$\int_{|x|>\eta} \int_{\R_+^d} \lp \left |xg_r(t)\right | \wedge 1\rp \dd t \nu ( \! \dd x)<+\infty\, .$$
Since $M=\sum_{i}\delta_{\tau_i} \delta_{Y_i}$,
$$\scal M \varphi = \sum_i Y_i \varphi(\tau_i)\, .$$
Now suppose $\varphi_n \to 0$ in $\Scd$. Then for large $n$, $|\varphi_n| \I g_r$, and
\begin{align*}
 |\scal M {\varphi_n} | = |\sum_i \varphi_n(\tau_i)Y_i | &\I \sum_i |Y_i| g_r(\tau_i) \\
 &=\int_{\R_+^d}\int_{|x|>\eta} M(\! \dd t , \dd x) g_r(t) |x| <+\infty \qquad \text{a.s.}
\end{align*}
For a.a. fixed $\omega\in \Omega$, $\varphi_n(\tau_i(\omega)) \to 0$ as $n\to +\infty$, $|\varphi_n (\tau_i(\omega))| \I g_r(\tau_i(\omega))$ and $\sum_i g_r(\tau_i(\omega)) |Y_i(\omega)|<+\infty$. By the dominated convergence theorem,
$$\scal{M(\omega)}{\varphi_n} =\sum_i \varphi_n(\tau_i(\omega)) Y_i(\omega) \to 0 \qquad \text{as} \ n\to +\infty \, .$$
Therefore, the linear functional $\varphi_n \mapsto \scal{M(\omega)}{\varphi_n}$ is continuous on $\Scd$, and so $M(\omega) \in \temd$ for a.a. $\omega \in \Omega$.

To prove \textit{(ii)}, we first prove that the Poisson integral on the right hand side of \eqref{def2} is well defined, and we will need the \textbf{PAM} condition. Let $\varphi \in \Scd$ and let $\Phi(t)=\int_{[t, +\infty[} \varphi(s) \dd s$. Then \eqref{def2} is well defined if 
\begin{equation}\label{condition2}
 \int_{|x|>\eta} \int_{\R_+^d} \lp \left |x\Phi(t)\right | \wedge 1\rp \dd t \, \nu( \! \dd x) <+\infty \, .
\end{equation}
Using \eqref{estimphi} in Lemma \ref{capitalphi} below, property \eqref{condition2} is established in the same way as \eqref{condition} and, as above, the right-hand side of \eqref{def2} defines almost surely a tempered distribution. Let $\varphi \in \Scd$. Then
\begin{align}\label{autre}
 \int_{\R_+^d}\int_{|x|>\eta}M( \! \dd t, \dd x) \int_{[t,+\infty[}\dd s \, \varphi(s)x
 =\sum_{i\s 1}\int_{\R_+^d}Y_i \mathds 1_{\tau_i\in [0,s]}\varphi(s) \dd s \, .
\end{align}
Following the argument in \eqref{measure}, we want to be able to use Fubini's theorem to exchange the sum and the integral in the last expression. For any $\alpha \in \N$, by the same argument as in the proof of Lemma \ref{capitalphi} below with $\beta=0$,
\begin{align*}
 \sup_{t\in \R_+^d}(1+|t|^\alpha)\int_{[t,+\infty[} |\varphi(s) |  \dd s \I C \m N_{|\alpha |+ 2d}(\varphi)\, .
\end{align*}
As in the proof of \eqref{condition}, we deduce that
\begin{equation*}\label{condition3}
 \int_{|x|>\eta} \int_{\R_+^d} \lp \left |x\int_{[t,+\infty[} |\varphi(s) | \dd s \right | \wedge 1\rp \dd t \, \nu( \! \dd x) <+\infty \, .
\end{equation*}
Then $\sum_{i\s 1}\int_{\R_+^d}|Y_i| \mathds 1_{\tau_i\in [0,s]}|\varphi(s)| \dd s<+\infty$, and by \eqref{autre} and Fubini's Theorem,
$$\int_{\R_+^d}\int_{|x|>\eta}M( \! \dd t, \dd x) \int_{[t,+\infty[}\dd s \varphi(s)x=\int_{\R_+^d}\sum_{i\s 1}Y_i \mathds 1_{\tau_i\in [0,s]}\varphi(s) \dd s = \int_{\R_+^d} P_s \varphi(s) \dd s \, .$$ 
This establishes \eqref{def2}. Property \textit{(iii)} now follows by replacing $\varphi$ by $\varphi^{(\mathbf 1_d)}$ in \eqref{def2}.
\end{proof}

\begin{lem}\label{capitalphi}
 For $\varphi \in \Scd$, let $\Phi$ be the function defined by $\Phi(t)=\int_{[t,+\infty)} \varphi(s) \dd s$. Let $p\in \N$, $\alpha, \beta \in \N^d$, such that $|\alpha|, |\beta |\I p$. Then for all $a\in \R^d$, there is $C=C(p,d,a)<+\infty$, such that, for all $\varphi \in \Scd$,
\begin{equation}\label{estimphi}
 \sup_{t\s a} \left |(1+ |t^\alpha|) \Phi^{(\beta)}(t) \right | \I  C' N_{p+2d}(\varphi) \, .
\end{equation}
\end{lem}
\begin{proof}
Let $t\in \R^d$. Then $\Phi(t)= \int_{\R^d_+} \varphi (s+t)  \dd s$, so $\Phi^{(\beta)}(t)=  \int_{\R^d_+} \varphi^{(\beta)}(s+t) \dd s$. Therefore,
\begin{align*}
 \left |(1+ |t^\alpha|) \Phi^{(\beta)}(t) \right | \I (1+|t^\alpha |) \int_{\R^d_+} \left | \varphi^{(\beta) }(s+t) \right | \dd s &= (1+|t^\alpha |) \int_{\R^d_+} \frac{\left | \varphi^{(\beta) }(s+t) \right | \lp 1+\left | (t+s)^{\alpha+\mathbf 2_d} \right | \rp}{1+\left | (t+s)^{\alpha+\mathbf 2_d} \right |  } \dd s \\
 &\I \m N_{p+2d}(\varphi)(1+|t^\alpha |) \int_{\R_+^d} \frac{1}{1+\left | (t+s)^{\alpha+\mathbf 2_d} \right | } \dd s \\
 & \I C N_{p+2d}(\varphi)
\end{align*}
for $t\s a$, where $C$ is a constant depending only on $p$, $d$ and $a$. 
\end{proof}

\subsection{Multidimensional L\'evy white noise: the general case}

\indent\indent
The following lemma extends to $d$-parameter L\'evy fields the property recalled in Remark \ref{equiv}.

\begin{lem}\label{timedep}
 Let $X$ be a $d$-parameter L\'evy field with characteristic triplet $(\gamma, \sigma, \nu)$ and let $\alpha>0$. The following are equivalent:
\textit{(i)} $\forall t\in \R_+^d, \ \E\lp |X_t|^\alpha\rp<+\infty$; \ \ \textit{(ii)} $\exists t\in \lp \R_+\backslash\{0\}\rp^d : \E\lp |X_t|^\alpha\rp<+\infty $; \ \ \textit{(iii)} $\displaystyle \int_{|x|>1}|x|^\alpha \nu( \! \dd x)<+\infty$\,.
\end{lem}

\begin{proof}
 Clearly, $\textit{(i)}$ implies $\textit{(ii)}$. Suppose that $\textit{(ii)}$ is true for some $t$ in $\lp \R_+\backslash \{0\}\rp^d$. By a previous discussion, the process $X^{i,t}$ obtained by fixing all coordinates of the parameter $t$ except the $i$-th is again a L\'evy process with characteristic triplet $(\gamma ,\sigma, \nu)\prod_{j\neq i}t_j$. By an application of \cite[Theorem 25.3]{sato} we deduce that $\prod_{j\neq i}t_i \int_{|x|>1}|x|^\alpha \nu( \! \dd x)<+\infty$ and then $\textit{(iii)}$ is verified. Suppose now that $\textit{(iii)}$ is true. Let $t\in \R_+^d$, and $1\I i\I d$. Since $\prod_{j\neq i}t_i \int_{|x|>1}|x|^\alpha \nu( \! \dd x)<+\infty$, another application \cite[Theorem 25.3]{sato} gives us $ \E\lp \left | X^{i,t}_s \right |^\alpha \rp <+\infty$ for all $s\in \R_+$. Since $i$ and $t$ are taken arbitrarily, we deduce $\textit{(i)}$.
\end{proof}

We need a technical lemma that essentially states that for a compound Poisson sheet $X^P$, there is a well-chosen sequence $\lp \varphi_n\rp_{n\s 1}$ of test-functions with suitably decreasing compact support such that $X^P$ is constant on $\text{supp}\, ( \varphi_n)$ for $n$ large enough (this was established in dimension one during the proof of Proposition \ref{growth}).
\begin{lem}\label{sequence}
 Let $X^P$ be a $d$-parameter L\'evy field with jump measure $J_X$ and characteristic triplet $(0, 0, \mathds 1_{|x| \s 1}\nu)$, where $\lambda :=\int_{|x| \s 1} \nu(\! \dd x) <+\infty$. Let $L$ be the compound Poisson process defined by $L_t=X^P_{(\mathbf 1_{d-1}, t)}$, and let $(S_n)_{n\s 1}$ denote its sequence of jump times. Then for all $p\in \N$, there exists a finite non random constant $C_p$ with the following property: for all $\omega \in \Omega$, there exists a sequence $\lp \varphi_n\rp_{n\s 1}$ of functions (depending on $\omega$) in $\m D(\R^d)$ such that
 \begin{equation}\label{grando}
 \mathcal N_p(\varphi_n)\mathds 1_{S_n\s 1} \I C_p S_n^{3d+4p} \mathds 1_{S_n\s 1} \, ,
\end{equation}
and there exists an event $\Omega'$ such that $\PR(\Omega')=1$ and for all $\omega \in \Omega'$, there exists an integer $N(\omega)$ such that, for all $n\s N(\omega)$, $X^P$ is constant on the support of $\varphi_n$ and
\begin{equation}\label{lsn}
 \scal{X^P}{\varphi_n}(\omega)=L_{S_n}(\omega) \, .
\end{equation}
\end{lem}

\begin{proof}
 As in the proof of Proposition \ref{growth}, we will construct a sequence $\lp \varphi_n\rp_{n\s 1}$ of functions with suitably decreasing compact support, and then use a Borel-Cantelli argument to show that $X^P$ is constant on this support. Let $\varphi \in \mathcal D(\R^d)$ with $\text{supp} \, \varphi \subset [0,\mathbf 1_{d}]$ and $\int_{\R^d}\varphi=1$. Similar to \eqref{phin}, the sequence $\lp\varphi_n\rp_{n\s 1}$ is defined by
$$ \varphi_n(t)=S_n^{3d}\varphi\lp (t_1-1)S_n^3,...,(t_{d-1}-1)S_n^3, (t_d-S_n)S_n^3 \rp\, , \ \ \ t\in \R^d \, ,$$
so that $\text{supp} \, \varphi_n \subset \left [ \lp \mathbf 1_{d-1} ,S_n\rp, \lp 1+\frac 1 {S_n^3},...,1+\frac 1 {S_n^3}, S_n+\frac 1 {S_n^3}\rp \right ]$ and $\int_{\R^d} \varphi_n=1$. Let $p\in \N$. Then
\begin{align*}
 \mathcal N_p(\varphi_n)\mathds 1_{S_n\s 1}&=\sum_{|\alpha|,|\beta| \I p} \sup_{t\in \R^d}\left | t^\alpha \varphi_n^{(\beta)}(t)\right | \mathds 1_{S_n\s 1} \\
 &=\sum_{|\alpha|,|\beta| \I p} \sup_{t\in \left [ 0, (2,...,2, S_n+1) \right ]} t ^\alpha \left |\varphi_n^{(\beta)}(t)\right | \mathds 1_{S_n\s 1} \\
 &\I \sum_{|\alpha|,|\beta| \I p} 2^{\sum_{i=1}^{d-1} \alpha_i} \lp S_n+1 \rp^{\alpha_d}\sup_{t\in \R^d} \left |\varphi_n^{(\beta)}(t)\right | \mathds 1_{S_n\s 1} \\ 
 &\I \sum_{|\alpha|,|\beta| \I p} 2^{\sum_{i=1}^{d-1} \alpha_i} \lp S_n+1 \rp^{\alpha_d} S_n^{3\lp d + \sum_{i=1}^d \beta_i \rp} \mathcal N_p (\varphi) \mathds 1_{S_n\s 1}\\ 
&\I C_p' \mathcal N_p(\varphi) S_n^{3d+4p}\mathds 1_{S_n\s 1} \, ,
\end{align*}
for some finite non random constant $C'_p$. Therefore \eqref{grando} holds and $C_p:=C'_p \m N_p(\varphi)$ depends only on $\varphi$ and $p$. Let 
$$I_{n,k}=\left ] \lp\mathbf 1_{d-1} ,S_n\rp, \lp1+\frac 1 {S_n^k},...,1+\frac 1 {S_n^k}, S_n+\frac 1 {S_n^k}\rp \right [ \, ,$$
and let $A_{n,k}$ be the event ``$X^P$ is constant in the box $I_{n,k}$''. Clearly, \eqref{lsn} holds on $A_{n,k}$. 

   Observe that
\begin{align*}
 \PR(A_{n,k}^c)&= \PR \left \{ X^P \ \text{has at least one jump time in the set} \ J_{n,k} \right \} \\
 &=\PR \left \{ J_{X^P}\lp \lp\R\backslash [-1,1]\rp \times J_{n,k} \rp  \s 1 \right \} \, ,
\end{align*}
where $J_{n,k}$ is defined as the following set:
\begin{align*}
 J_{n,k} &=\left [ \mathbf 0_{d} , \lp 1+\frac 1 {S_n^k},...,1+\frac 1 {S_n^k}, S_n+\frac 1 {S_n^k}\rp \right [ \backslash \left [ \mathbf 0_d , \lp \mathbf 1_{d-1},S_n\rp \right ]=J_{n,k}^1 \cup J_{n,k}^2 \, ,
\end{align*}
where $J_{n,k}^1$ and $J_{n,k}^2$ are disjoint sets defined by
\begin{align*}
 J_{n,k}^1 &=\left \{ x\in \R_+^d : \forall \, 1\I i\I d-1, \, x_i<1+\frac 1 {S_n^k}, \ x_d\I S_n, \ \text{and} \ \exists i_0 \in \left\{1,...,d-1\right \} \ \text{s.t.} \ x_{i_0}>1 \right\}\, , \\
J_{n,k}^2 &=\left ] \lp \mathbf 0_{d-1},S_n \rp ,\lp 1+\frac 1 {S_n^k},...,1+\frac 1 {S_n^k},S_n+\frac 1 {S_n^k} \rp \right [\, .
\end{align*}

 Therefore we can write
\begin{equation}\label{dec1}
\begin{aligned}
   \PR(A_{n,k}^c)&=\PR \left \{ J_{X^P}\lp \lp\R\backslash  [-1,1]\rp \times J_{n,k}^1 \rp+J_{X^P}\lp \lp\R\backslash [-1,1]\rp \times J_{n,k}^2 \rp  \s 1 \right \} \\
  &\I \PR \left \{ J_{X^P}\lp \lp\R\backslash [-1,1]\rp \times J_{n,k}^1 \rp \s 1 \right \} +\PR \left \{ J_{X^P}\lp \lp\R\backslash  [-1,1]\rp \times J_{n,k}^2 \rp\s 1 \right \} \, .
\end{aligned}
\end{equation}
Let $\mathcal F_{(\mathbf 1_{d-1}, t)}=\sigma \lp X_s, \ s\in \left [\mathbf 0_{d}, \lp \mathbf 1_{d-1}, t\rp \right ] \rp$ and $\mathcal F_{(\mathbf 1_{d-1}, \infty)}=\bigvee_{t\in \R_+} \mathcal F_{(\mathbf 1_{d-1}, t)}$. We also write $H_1=\left \{ x\in \R_+^d : x_1\I 1,..., x_{d-1}\I 1\right \}$. Then due to the independence of the increments of $X^P$, the family of random variables $\lp J_{X^P}\lp \lp\R\backslash  [-1,1]\rp \times A \rp \rp_{A\subset \R_+^d \backslash H_1}$ is independent of $\mathcal F_{(\mathbf 1_{d-1}, \infty)}$. Since $S_n$ is $\mathcal F_{(\mathbf 1_{d-1}, \infty)}$-measurable, we deduce that conditionally on $S_n$, the random variable $J_{X^P}\! \lp \lp\R\backslash [-1,1]\rp \times J_{n,k}^1 \rp$ has a Poisson law with parameter $\lambda\leb{d}{J_{n,k}^1}$, where $\lambda :=\int_{|x|>1}\nu(d x)$. Further, on the event $\left \{ S_n\s 1\right \}$,
\begin{align*}
 \leb{d}{J_{n,k}^1} &= \sum_{j=1}^{d-1} \binom{d-1}{j} S_n \lp \frac 1 {S_n^k} \rp^j \lp 1+\frac 1 {S_n^k}\rp ^{d-1-j} 
 \I 3^{d-1}  S_n^{-(k-1)}\, .
\end{align*}
Indeed, the Lebesgue measure of a subset of $J_{n,k}^1$ of vectors with exactly $j$ components strictly greater than one is $S_n \lp \frac 1 {S_n^k} \rp^j \lp 1+\frac 1 {S_n^k}\rp ^{d-1-j}$, and there are $\binom{d-1}{j}$ such subsets. We deduce that
\begin{align}\label{dec6}
 \PR \left \{ J_{X^P}\lp \lp\R\backslash[-1,1]\rp \times J_{n,k}^1 \rp \s 1 \right \} &\I \PR\left \{S_n\I 1\right \} + \E \lp \mathds 1_{S_n>1}\lp 1-e^{-\lambda\leb{d}{J_{n,k}^1}} \rp \rp \nonumber \\
 &\I \PR\left \{S_n\I 1\right \} +\lambda 3^{d-1} \E\lp  S_n^{-(k-1)} \rp \, .
\end{align}
We also define a process $\tilde L_t=X^P_{(\mathbf 2_{d-1},t)}$. It is a L\'evy process with L\'evy measure $\mu( \! \dd x)=2^{d-1}\mathds 1_{|x|>1}\nu( \! \dd x)$. Since $X^P$ is piecewise constant, $\tilde L$ is a piecewise constant L\'evy process, therefore a compound Poisson process (see \cite[Theorem 21.2]{sato}). On the event $\left \{S_n >1 \right \}$, we have $J_{n,k}^2 \subset [(\mathbf 0_{d-1}, S_n) , (\mathbf 2_{d-1}, S_n+ S_n^{-k})]$. Therefore if $X^P$ has a jump point in $J_{n,k}^2$ then $\tilde L$ has a jump in $\left ] S_n, S_n+S_n^{-k}\right [$.  Let $\mathcal G_t=\sigma\lp X_u : u\in [0, (\mathbf 2_{d-1},t] \rp$. Then $S_n$ is a $\mathcal G$-stopping time and $\tilde L$ is a L\'evy process adapted to the filtration $\mathcal G$, so by the strong Markov property, the number of jumps of the process $(\hat L_t)_{t\s 0}=\lp \tilde L_{t+S_n}-\tilde L_{S_n}\rp_{t\s 0}$ is independent of $S_n$ and has Poisson distribution of parameter $2^{d-1}\lambda t$. Therefore we can write
\begin{align}\label{dec2}
 \PR \left \{ J_{X^P}\lp \lp\R\backslash [-1,1] \rp \times J_{n,k}^2 \rp\s 1 \right \} &\I \PR \left \{ S_n \I 1  \right \}+ \PR \lp \left \{ J_{X^P}\lp \lp\R\backslash [-1,1] \rp \times J_{n,k}^2 \rp\s 1 \right \} \cap \left \{ S_n >1 \right \}\rp \nonumber\\
 &\I \PR \left \{ S_n \I 1  \right \}  + \PR \left \{ \tilde L \ \text{has a jump in} \ \lp S_n, S_n + \frac 1 {S_n^k} \rp \right \}   \nonumber \\
 & = \PR\left \{ S_n \I 1  \right \} + \PR \left \{ \hat L \ \text{has a jump in} \ \lp 0, \frac 1 {S_n^k} \rp \right \}   \nonumber \\
 &= \PR \left \{ S_n \I 1  \right \}+ \E\lp 1- \exp \left [-\frac{2^{d-1}\lambda}{S_n^k}\right ]  \rp  \nonumber \\
 & \I \PR \left \{ S_n \I 1  \right \}+ \E\lp \frac{2^{d-1}\lambda}{S_n^k} \rp  \, .
\end{align}
Using the density of the Gamma distribution, we see that
\begin{equation}\label{dec4}
 \PR\left \{ S_n \I 1  \right \}= \int_0^1 \frac{\lambda^n}{(n-1)!}e^{-\lambda x}x^{n-1} \dd x \I \frac{\lambda^n}{(n-1)!} \, .
\end{equation}
Integrating the Laplace transform of $S_n$ as in \eqref{gammaint}, for $n\s 4$, we see that
\begin{equation}\label{gammaintd}
  \E\lp   S_{n}^{-3}\rp = \frac{\lambda^3}{(n-1)(n-2)(n-3)} \qquad \text{and} \qquad \E\lp  S_{n}^{-2}\rp = \frac{\lambda^2}{(n-1)(n-2)} \, .
\end{equation}
Then we get from \eqref{dec1},\eqref{dec6}, \eqref{dec2} with $k=3$, \eqref{dec4} and \eqref{gammaintd}, that for $n\s 4$:
\begin{align*}
  \PR\lp A_{n,3}^c\rp & \I \frac{2\lambda^n}{(n-1)!} +\lambda 3^{d-1} \E\lp \frac 1 {S_n^2} \rp +  \lambda 2^{d-1}\E\lp \frac {1} { S_{n}^3} \rp  \\
  &=\frac{2\lambda^n}{(n-1)!} + \frac{\lambda^2 3^{d-1}}{(n-1)(n-2)}+  \frac{\lambda^3 2^{d-1}}{(n-1)(n-2)(n-3)}  \, ,
\end{align*}
and we deduce that $\sum_{n\s 1}\PR \lp A_{n,3}^c\rp <\infty$. By the Borel-Cantelli Lemma, 
\begin{equation}
 \PR\lp \underset{n\to +\infty}{\lim\sup} A_{n,3}^c \rp=0 \, ,
\end{equation}
and the set $\Omega'=\underset{n\to +\infty}{\lim\inf} A_{n,3}$ has probability one. This completes the proof.
\end{proof}

We now return to the question of whether or not a L\'evy white noise is a tempered distribution. Similar to \eqref{omega_set_dot}, for any $d$-dimensional L\'evy noise $\dot X$, we define  the set $\Omega_{\dot X}$ by
\begin{align}\label{omega_set_dotd}
 \Omega_{\dot X}= \left \{ \omega \in \Omega : \dot X(\omega) \in \temd \right \} \, ,
\end{align}
 and we have the following characterization.
\begin{theo}\label{temperedd}
 Let $X$ be a $d$-parameter L\'evy field with jump measure $J_X$ and characteristic triplet $(\gamma, \sigma, \nu)$ and $\dot X$ the associated L\'evy white noise. Then the following holds for the set $\Omega_{\dot X}$ defined as in \eqref{omega_set_dotd}:
\begin{itemize}
 \item[(i)] If there exists $\eta>0$ such that $\E\lp |X_{\mathbf 1_d}|^\eta\rp<+\infty$, then $\PR\lp \Omega_{\dot X}\rp=1$.
 \item[(ii)] If for all $\eta>0, \ \E\lp |X_{\mathbf 1_d}|^\eta\rp=+\infty$, then $\PR\lp \Omega_{\dot X}\rp=0$.
\end{itemize}
\end{theo}

\begin{rem}By Lemma \ref{timedep}, the equivalent condition mentioned in Remark \ref{equiv} remains valid in the $d$-parameter case.
\end{rem}

As mentioned in the Introduction, the first assertion of Theorem \ref{temperedd} was established in \cite[Theorem 3]{fageot} using a different definition of L\'evy white noise. In Proposition \ref{lnoise} below, we show that the two definitions are equivalent.

\begin{proof}[Proof of Theorem \ref{temperedd}]
 To prove \textit{(i)}, by the L\'evy-Itô decomposition (Theorem \ref{lid}), Corollary \ref{martingalecord} and Lemma \ref{pnoise} \textit{(ii)}, we have $\PR\lp \Omega_X\rp=1$. Since derivation maps $\temd$ to itself, we deduce that $\PR\lp \Omega_{\dot X}\rp=1$.

 To prove \textit{(ii)}, suppose that $\dot X$ does not have a \textbf{PAM}. We can use Theorem \ref{lid} to decompose $X$ into the sum of a continuous part $C$, a small jumps part $X^M$ and a compound Poisson part $X^P$. By Corollary \ref{martingalecord}, $\PR \lp \Omega_{C+X^M}\rp=1$. Then we deduce that for all $\omega \in \Omega_{\dot X} \cap \Omega_{C+X^M}$, $\dot X^P(\omega)=\dot X(\omega)-\dot C(\omega)-\dot X^M(\omega)=\dot X(\omega)-\lp C(\omega)+X^M(\omega)  \rp^{(\mathbf 1_d)}$ belongs to $\temd$. The general strategy of the proof is to construct, from the compound Poisson sheet $X^P$, a compound Poisson process that has the same moment properties, and show that when $\dot X^P\in \temd$, this process has polynomial growth at infinity, and this occurs with probability zero by Proposition \ref{longterm}\textit{(ii)}.           
 
 We first examine the noise $\dot X^P$ associated with the compound Poisson part. The jump measure $J_{X^P}( \! \dd s ,\dd x )=\mathds 1_{|x|>1}J_X( \! \dd s ,\dd x)$ of $X^P$  is a Poisson random measure on $ \R_+^d \times \lp\R\backslash \{0\}\rp$  and $J_{X^P}=\sum_{i\s 1}\delta_{\tau_i}\delta_{Y_i}$, where $\tau_i\in \R_+^d$ and $|Y_i|\s 1$. By Lemma \ref{fubinicompact}, for all $\varphi\in \m D(\R^d)$,
\begin{align}\label{eq1}
 \scal{\dot X^P}{\varphi}= \int_{\R_+^d}\int_{|x|>1} x \varphi(t) J_X(\! \dd t , \! \dd x) =\sum_{i\s 1} Y_i \varphi(\tau_i)\ .
\end{align}
By Lemma \ref{sequence}, for all $\omega \in \Omega$, there exists a sequence $(\varphi_n)_{n\s 1}(\omega)$ of smooth compactly supported functions such that \eqref{grando} holds. Furthermore, there is an event $\Omega'\subset \Omega$ with probability one such that there is an integer $N(\omega)$ with the property that for all $n\s N(\omega)$, $X^P$ is constant on the support of $\varphi_n (\omega)$, and \eqref{lsn} holds. Let $L$ be the compound Poisson process defined in Lemma \ref{sequence} by $L_t=X^P_{(\mathbf 1_{d-1}, t)}$. We restrict ourselves to $\omega \in \Omega_{\dot X}\cap \Omega_{C+X^M}\cap \Omega'$, but we drop the dependence on $\omega$ in the following for simplicity of notation. We write $\Phi_n(t)=\int_{[t, +\infty)} \varphi_n(s) \dd s$. Let $\theta \in C^\infty (\R^d)$ be such that $\theta=0$ on the set $\left \{ t\in \R^d : t_1\wedge...\wedge t_d \I -1\right \}$ and $\theta=1$ on the set $\left \{ t\in \R^d : t_1\wedge...\wedge t_d \s -\frac 1 2 \right \}$ and such that all its derivatives are bounded. Then for all $n\s 1, \theta\Phi_n \in \mathcal D(\R^d) \subset \Scd$. So, in particular, for all $n\s 1$, since $\theta$ is constant on $\R_+^d$,
\begin{align*}
 \scal{\dot X^P}{\theta\Phi_n} =(-1)^d \scal{ X^P}{\lp \theta\Phi_n\rp^{(\mathbf 1_d)}}=(-1)^d\scal{ X^P}{\lp \Phi_n\rp^{(\mathbf 1_d)}} = \scal{X^P}{\varphi_n} = L_{S_n} \, ,
\end{align*}
by \eqref{lsn}, and since $\Omega_{\dot X}\cap \Omega_{C+X^M} \subset  \Omega_{\dot X^P}$, we deduce that 
\begin{align}\label{i1}
  \left | L_{S_n} \right | \I C\mathcal N_p(\theta \Phi_n)\, ,
\end{align}
for some real number $C$ and integer $p$ (both depending on $\omega$).  For $\alpha, \beta\in \N^d$, with $|\alpha|, |\beta|\I p$, we estimate $ \sup_{t\in \R^d} \left | t^\alpha  \lp \theta \Phi_n \rp^{(\beta)} \right |$. Since all the derivatives of $\theta$ are bounded,
\begin{align*}
  \sup_{t\in \R^d} \left | t^\alpha  \lp \theta \Phi_n \rp^{(\beta)}(t) \right | = \sup_{t\s -\mathbf 1_d} \left | t^\alpha  \lp \theta \Phi_n \rp ^{(\beta)}(t)\right |
 & = \sup_{t\s -\mathbf 1_d} \left | t^\alpha  \sum_{\gamma \I \beta } \binom{\beta}{\gamma}  \Phi_n^{(\gamma)}(t)\theta ^{(\beta-\gamma)}(t) \right | \\
  &\I C_1 \sum_{\gamma \I \beta } \sup_{t\s -\mathbf 1_d} \left | t^\alpha    \Phi_n^{(\gamma)}(t) \right | \, ,
\end{align*}
for some constant $C_1$ depending only on $p$ and $\theta$. By \eqref{estimphi}, for some constant $C_2$,
\begin{align*}
\sup_{t\s -\mathbf 1_d} \left | t^\alpha    \Phi_n^{(\gamma)}(t) \right |  \mathds 1_{S_n\s 1} \I C_2 \m N_{p+2d}(\varphi_n)  \mathds 1_{S_n\s 1}  \I C_3 S_n^{\tilde p} \mathds 1_{S_n\s 1} \, ,
\end{align*}
by \eqref{grando}, for some constant $C_3$ and $\tilde p$ independent of $n$. Therefore, for any integer $p$, there is an integer $\tilde p$ and a constant $C$ depending only $p$ and $d$, such that
\begin{align}\label{i2}
 \mathcal N_p(\theta \Phi_n) \mathds 1_{S_n\s 1} \I C S_n^{\tilde p} \mathds 1_{S_n\s 1}\, .
\end{align}
We deduce from \eqref{i1} and \eqref{i2} that
\begin{align*}
 \left |\frac{L_{S_n}}{S_n^{\tilde p}}\right | \mathds 1_{S_n\s 1} \I C \mathds 1_{S_n\s 1}<+\infty\, .
\end{align*}
As in the proof of Proposition \ref{growth}, we deduce that for all $\omega \in \Omega_{\dot X}\cap \Omega_{C+X^M} \cap \Omega'$, there exists $  p(\omega) \in \N$ and $C(\omega) \in \R_+$ such that
\begin{align}\label{ineg1}
 \underset{t\to +\infty}{\lim\sup} \frac{|L_{t}|(\omega)}{1+t^{\tilde p(\omega)}}\I C(\omega)<+\infty \, .
\end{align}
Since $L$ is a compound Poisson process with no absolute moment of any positive order (it has the same L\'evy measure as $X^P$) we can now conclude by Proposition \ref{longterm}\textit{(ii)} that  $ \Omega_{\dot X}  \cap \Omega_{C+X^M} \cap \Omega'$ is contained in a set of probability zero. Since $\PR \lp \Omega_{C+X^M} \cap \Omega'\rp=1$, we deduce that $\PR\lp  \Omega_{\dot X} \rp=0$.
\end{proof}

As a consequence of Theorem \ref{temperedd}, we get the following result.
\begin{cor}\label{cor3.15}
 Let $X$ be a $d$-parameter L\'evy field with jump measure $J_X$ and characteristic triplet $(\gamma, \sigma ,\nu)$, and let $\Omega_X$ be the set defined as in \eqref{omega_setd}. 
 \begin{itemize}
 \item[(i)] If there exists $\eta>0$ such that $\E\lp |X_1|^\eta\rp<+\infty$, then $\PR(\Omega_X)=1$.
 \item[(ii)] If for all $\eta>0, \ \E\lp |X_1|^\eta\rp=+\infty$, then $\PR(\Omega_X)=0$.
 \end{itemize}
\end{cor}

\begin{proof}
 Property \textit{(i)} follows immediately from Corollary \ref{martingalecord} and Lemma \ref{pnoise}. To prove \textit{(ii)}, by the fact that the derivative of a tempered distribution is a tempered distribution, $\Omega_{X}\subset \Omega_{\dot X}$. By Theorem \ref{temperedd}\textit{(ii)}, we conclude that $\PR(\Omega_X)=0$.
\end{proof}

\begin{rem}
 The statement of Corollary \ref{measurable} extends directly to $d$-dimensional L\'evy white noise, with the same proof.
\end{rem}

We now relate our definition of L\'evy white noise (Definition \ref{levy_noise}) to stochastic integrals, and to \cite[Theorem 3]{fageot}.

\begin{prop}\label{lnoise} Let $\dot X$ be a L\'evy white noise with jump measure $J_X$ and characteristic triplet $(\gamma, \sigma, \nu)$ that has a \textbf{PAM}. 
\begin{itemize}
\item[(i)] For all functions $\varphi \in \Scd$, we have the following equality:
\end{itemize}
\vspace{-1em}
\begin{equation}\label{intsto}
\begin{aligned}
   \scal{\dot X}{\varphi}&=\int_{\R_+^d}\varphi(t) \dd X_t  \\
  &:= \gamma \int_{\R^d_+}\varphi(t)\dd t+ \sigma \int_{\R^d_+}\varphi(t) \dd W_t\\
  &\hspace{1cm} +\int_{\R^d_+}\int_{|x|\s1} x\varphi(t) J_X( \! \dd x, \dd t)+ \int_{\R^d_+}\int_{|x|<1}x\varphi(t) \lp J_X( \! \dd x ,\dd t) - \nu( \! \dd x)\dd t \rp \\
  &=\gamma A_1(\varphi)+ \sigma A_2(\varphi) + A_3(\varphi)+ A_4(\varphi) \, , 
\end{aligned}
\end{equation}
\begin{itemize}
\item[]where the second integral is a Wiener integral (cf. Remark \ref{rem_stochastic_integral}), and the last two are Poisson integrals as defined in \cite[Lemma 12.13]{kallenberg}.
 \item[(ii)] The characteristic functional of the L\'evy white noise is given, for all $\varphi \in \Scd$, by
\begin{align*}
 \E\lp e^{i\scal{\dot X}{\varphi}} \rp = \exp \left[ \int_{\R_+^d} \psi\lp \varphi (t) \rp \dd t \right ]\, ,
\end{align*}
where $\psi$ is the L\'evy symbol of $X$:
$$\psi(z)= i\gamma z-\frac 1 2 \sigma^2 z^2+\int_\R \lp e^{ixz}-1-izx\mathds 1_{|x|\I1}\rp \nu( \! \dd x)\, .$$
\end{itemize}
\end{prop}

\begin{proof}
 Even without \textbf{PAM}, equality \eqref{intsto} has already been proven in Lemma \ref{fubinicompact} when $\varphi\in \m D(\R^d)$. We now assume that $X$ has a \textbf{PAM} and check first that for $\varphi \in \Scd$, the right-hand side of \eqref{intsto} is well-defined. Since $\Scd\subset L^1(\R^d)\cap L^2(\R^d)$, this is clearly the case for $A_1(\varphi)$ and $A_2(\varphi)$. For $A_3(\varphi)$, using \pam, one checks condition \eqref{condition} as in the proof of Lemma \ref{pnoise}. For $A_4(\varphi)$, one checks condition \eqref{small_jumps_condition} using the same proof as when $\varphi\in \m D(\R^d)$.
 
 We now deduce \eqref{intsto} for $\varphi\in \Scd$ (assuming \textbf{PAM}). By definition,
\begin{align*}
 \scal{\dot X}{\varphi}&= (-1)^d\left [\gamma \int_{\R_+^d} \lp \prod_{i=1}^d t_i \rp \varphi^{(\mathbf 1_d)}(t) \dd t +\sigma  \int_{\R_+^d} W_t \varphi^{(\mathbf 1_d)}(t) \dd t \right. \\
 & \hspace{3cm}+\left.   \int_{\R_+^d} X_t^P\varphi^{(\mathbf 1_d)}(t) \dd t+\int_{\R_+^d} X_t^M \varphi^{(\mathbf 1_d)}(t) \dd t \right ] \\
 &=  \gamma \tilde A_1(\varphi)+ \sigma \tilde A_2(\varphi)+\tilde A_3(\varphi)+\tilde A_4(\varphi)\, .
\end{align*}
The equality $A_3(\varphi)=\tilde A_3(\varphi)$ comes from Lemma \ref{pnoise}. For the other three terms, since $\m D(\R^d)$ is dense in $\Scd$, it suffices to check that $\varphi \mapsto \gamma \tilde A_1(\varphi)+ \sigma \tilde A_2(\varphi)+\tilde A_4(\varphi)$ and $\varphi \mapsto \gamma  A_1(\varphi)+ \sigma  A_2(\varphi)+ A_4(\varphi)$ define continuous (in probability) linear functionals of $\varphi \in \Scd$. For the first, this is obvious because $\dot X\in\temd$ by Theorem \ref{temperedd}. For the second, consider $(\varphi_n)\subset \Scd$ such that $\varphi_n\to 0$ in $\Scd$, hence in $L^1(\R^d)$ and $L^2(\R^d)$. Then $A_1(\varphi_n) \to 0$ and $A_2(\varphi_n)\to 0$ in probability. According to \cite[(2.34) p.27]{cohen},
$$\E\lp \exp \lp iA_4(\varphi)\rp\rp=\exp \left [\int_{\R^d_+}  \int_{|x|<1}\lp e^{ix\varphi(t)}-1-ix\varphi(t) \rp \dd t \nu( \! \dd x) \right ]\, ,$$
for all $\varphi \in \Scd$. By the inequality in \cite[Lemma 5.14]{kallenberg}, 
$$ \left |\int_{\R^d_+}  \int_{|x|<1}\lp e^{ix\varphi_n(t)}-1-ix\varphi_n(t) \rp \dd t \nu( \! \dd x) \right | \I \frac 1 2 \int_{\R^d_+}\varphi_n(t)^2 \dd t \int_{|x|<1}x^2 \nu( \! \dd x)\lims n 0\, .$$
So $A_4$ defines a linear functional on $\Scd$ that is continuous in law at $0$, hence continuous in probability. This completes the proof of \textit{(i)}.

To prove \textit{(ii)}, we use \textit{(i)} and standard results on the characteristic function of Poisson and Wiener integrals. See \cite[Lemma 12.2]{kallenberg} and \cite[(2.34) p.27]{cohen} for the Poisson integrals and \cite[Theorem 1.4.1]{multi} for the Wiener integral.
\end{proof}

\begin{rem}
 In dimension one, we used the map $I$ in Remark \ref{functional} to give an alternate proof of Theorem \ref{tempered}\textit{(ii)}. The analog of this map $I$ in higher dimensions also exists. Let $\theta\in \m D(\R)$ such that $\theta\s 0, \ \emph{supp} \, \theta \subset [0,1]$ and $\int_\R \theta=1$. We write $\tilde \theta= \theta \otimes \dots \otimes \theta$ the $d^\text{th}$-order tensor product of $\theta$ with itself: $\tilde \theta(s_1,\dots ,s_d)=\theta(s_1)\cdots\theta(s_d)$. Let $\varphi\in \Scd$. Define
 \begin{align}\label{f1}
 I_d \varphi (t)= \int_{(-\infty, t]} \! \dd s \int_{\R^d} \! \dd r \tilde{\Delta} _r^s \lp \varphi , \tilde \theta \rp \, ,
\end{align}
where 
\begin{align*}
 \tilde{\Delta} _r^s \lp \varphi , \theta \rp =\sum_{\varepsilon \in \{0,1\}^d}(-1)^{|\varepsilon|} \varphi (c_\varepsilon (r,s)) \tilde \theta (c_{1-\varepsilon}(r,s)) \, ,
\end{align*}
and $c_\varepsilon (r,s)$ was defined just after \eqref{increment}. It is easy to see that if $\varphi= \varphi_1 \otimes \dots \otimes \varphi_d$, where $\varphi_1, \dots , \varphi_d \in \Sc$, then $I_d \varphi= \lp I_1 \varphi_1 \rp \otimes \dots \otimes \lp I_1\varphi_d \rp$, where $I_1$ coincides with the map $I$ of Remark \ref{functional}. Then, since $I$ was built as an antiderivative, for such $\varphi$,
\begin{equation}\label{idfunctional}
 I_d\lp \frac{\partial^d \varphi}{\partial t_1 \cdots \partial t_d}\rp=\varphi \, .
\end{equation}
We have already shown that $I_1$ maps continuously $\Sc$ to itself. We equip $\Sc \otimes \cdots \otimes \Sc$ with the topology $\pi$ generated by the family of semi-norms $\mathcal N_{p_1, \dots , p_d}(\varphi_1 \otimes \cdots \otimes \varphi_d)=\prod_{i=1}^d \m N_{p_i} (\varphi_i)$. Then $I_d: \Sc \otimes \cdots \otimes \Sc \to \Sc \otimes \cdots \otimes \Sc$ is continuous (and then uniformly continuous by linearity). We denote $\Sc \hat \otimes_\pi \cdots \hat \otimes_\pi \Sc$ the completion of $\Sc \otimes \cdots \otimes \Sc$. By \cite[Theorem 51.6]{treves}, $\Sc \hat \otimes_\pi \cdots \hat \otimes_\pi \Sc\simeq \Scd$, therefore $I_d$ extends (by uniform continuity) to a continuous linear map from $\Scd$ to itself. Formula \eqref{idfunctional} is true by linearity for $\varphi \in \Sc \otimes \cdots \otimes \Sc$. Let $\varphi \in \Scd$. There is a sequence $(\varphi_n)_{n\s 1}$ of elements of $\Sc \otimes \cdots\otimes \Sc$ such that $\varphi_n \to \varphi$ in $\Scd$. Since derivation is a continuous map from $\Scd$ to itself, we deduce that \eqref{idfunctional} holds for any $\varphi \in \Scd$. 
\end{rem}

%\bibliography{../../biblio}
%\bibliographystyle{../../myplain}
\end{document}